\documentclass[letterpaper, 11pt]{article} 

\usepackage[english]{babel}
\usepackage{amsmath,amsfonts,amsthm}
\usepackage{wasysym}
\usepackage[dvips]{graphicx}
\usepackage{enumerate}
\usepackage{mathtools}
\usepackage{mathrsfs}
\usepackage{bbm}
\usepackage{xcolor}
\usepackage[pdfpagemode=UseNone,bookmarksopen=false,colorlinks=true,urlcolor=blue,citecolor=blue,citebordercolor=blue,linkcolor=blue]{hyperref}
\usepackage{enumitem}

\usepackage[colorinlistoftodos, textsize=tiny]{todonotes}

\makeatletter
\providecommand\@dotsep{5}
\renewcommand{\listoftodos}[1][\@todonotes@todolistname]{%
  \@starttoc{tdo}{#1}}
\makeatother

\usepackage[margin=1in]{geometry}

\usepackage[numbers]{natbib}

\newcommand{\eps}{\varepsilon}

\newcommand{\Nat}{\mathbb{N}}
\newcommand{\Real}{\mathbb{R}}

\newcommand{\ex}[1]{\mathbb{E}\left[#1\right]}
\newcommand{\Pa}{{\mathcal P}}

\newcommand{\B}{{\mathcal B}}

\newcommand{\E}{{\mathcal E}}

\newcommand{\F}{{\mathcal F}}
\newcommand{\M}{{\mathcal M}}
\newcommand{\C}{{\mathcal C}}

\newcommand{\X}{{\mathcal X}}

\newcommand{\T}{{\mathcal T}}
\newcommand{\R}{{\mathcal R}}

\newcommand{\Aux}{{\mathcal G}}
\newcommand{\Li}{{\mathcal L}}

\newcommand{\eul}{\mathrm{e}}

\newtheorem{firsttheorem}{Proposition}

\newtheorem{theorem}[firsttheorem]{Theorem}
\newtheorem{lemma}[firsttheorem]{Lemma}
\newtheorem{corollary}[firsttheorem]{Corollary}

\newtheorem{definition}[firsttheorem]{Definition}
\newtheorem{proposition}[firsttheorem]{Proposition}

\numberwithin{equation}{section}
\numberwithin{firsttheorem}{section}

\numberwithin{secondtheorem}{section}

\newcommand{\card}[1]{\left|#1\right|}
\newcommand{\pr}[1]{\mathrm{Pr}\left[#1\right]}

\newcommand{\boldrm}[1]{\boldsymbol{\mathrm{#1}}}
\newcommand{\pois}[1]{\mathrm{Po}\left(#1\right)}

\newcommand{\bigO}[1]{\mathcal{O}\left(#1\right)}

\newcommand{\e}[1]{\exp\left(#1\right)}

\newcommand{\scs}[1]{\text{{\scshape#1}}}

\DeclarePairedDelimiter\floor{\lfloor}{\rfloor}

\begin{document}
\title{Asymptotic Enumeration and Limit Laws for Multisets: the Subexponential Case}
\author{Konstantinos Panagiotou\thanks{Department of Mathematics, Ludwig-Maximilians-Universit\"at M\"unchen. E-mail: kpanagio@math.lmu.de.}\, and Leon Ramzews\thanks{Department of Mathematics, Ludwig-Maximilians-Universit\"at M\"unchen. E-mail: ramzews@math.lmu.de. Funded by the Deutsche Forschungsgemeinschaft (DFG, German Research Foundation), Project PA 2080/3-1.}}
\date{ }
\maketitle
	
%

\begin{abstract}
For a given combinatorial class $\C$ we study the class $\Aux = \scs{Mset}(\C)$ satisfying the multiset construction, that is, any object in $\Aux$ is uniquely determined by a set of $\C$-objects paired with their multiplicities. For example, $\scs{Mset}(\mathbb{N})$ is (isomorphic to) the class of number partitions of positive integers, a prominent and well-studied case. The multiset construction appears naturally in the study of unlabelled objects, for example graphs or various structures related to number
partitions. Our main result establishes the asymptotic size of the set $\Aux_{n,N}$ that contains all multisets in $\Aux$ having size $n$ and being comprised of $N$ objects from $\C$, as $n$ \emph{and} $N$ tend to infinity and when the counting sequence of $\C$ is governed by subexponential growth; this is a particularly important setting in combinatorial applications. Moreover, we study the component distribution of random objects from $\Aux_{n,N}$ and we discover a phenomenon that we baptise \emph{extreme condensation}: taking away the largest component as well as all the components of the smallest possible size, we are left with an object which converges in distribution as $n,N\to\infty$. The distribution of the limiting object is also retrieved.
Moreover and rather surprisingly, in stark contrast to analogous results for labelled objects, the results here hold uniformly in $N$.
\end{abstract}

\section{Introduction \& Main Results}
\label{sec:intro}

Let $\C$ be a combinatorial class, that is, a countable set endowed with a size function $\lvert\cdot\rvert:\C\to\Nat$ such that $\C_n := \{C\in\C:\lvert C\rvert = n\}$ contains only finitely many objects for all $n\in\Nat$. Then the class of $\C$-multisets  $\Aux = \scs{Mset}(\C)$ consists of all objects of the form
\[
	\big\{(C_1,d_1),\dots,(C_k,d_k)\big\},
	\quad
	k\in\Nat, ~~ C_i\in\C, d_i\in\Nat \text{ for all }  1\le i\le k,
\]
where $(C_i)_{1\le i\le k}$ are pairwise distinct and $d_i$ is the multiplicity of the object $C_i$ in the multiset. In simple words, a $\C$-multiset is a finite unordered collection of elements from $\C$ such that multiple occurrences of each element are admissible. For example, if $\C = \mathbb{N}$, then $\scs{Mset}(\C)$ contains all partitions of natural numbers, a prominent object. The multiset construction is omnipresent in  combinatorial settings, for example when $\C$ is some class of connected unlabelled graphs; this makes $\Aux$ the class of unlabelled graphs having connected components in $\C$. For many historical references and examples we refer the reader to the excellent books \cite{Flajolet2009,Leroux1998}. An alternative and instructive way to describe multisets of size $n\in\Nat$ is to make the connection to number partitions explicit as follows. First, choose a number partition of $n$. Then, assign to each of the parts an element of that size from $\C$. Hence, multisets are also called \emph{weighted integer partitions}, frequently encountered in the context of statistical physics of ideal gas. There, $c_k:=\lvert\{C\in\C:\lvert C\rvert = k\}\rvert$ describes the different possible states of a particle at energy level $k\in\Nat$, see \cite{Vershik1996} for a thorough overview.

Given $G=\{(C_1,d_1),\dots,(C_k,d_k)\} \in \Aux$ we denote by $\lvert G\rvert := \sum_{1\le i\le k}d_i\lvert C_i \rvert$ the \emph{size} and by $\kappa(G) := \sum_{1\le i\le k}d_i$ the number of \emph{components} of $G$. We further set
\[
	\Aux_n:=\{G\in\Aux:\lvert G\rvert = n\} \quad \text{and} \quad \Aux_{n,N} := \{G\in\Aux_n:\kappa(G)=N\}, \qquad n,N\in\Nat.
\]
Additionally, we define $\mathsf{G}_n$ and $\mathsf{G}_{n,N}$ to be   multisets drawn uniformly at random from $\Aux_n$ and $\Aux_{n,N}$, respectively.

A vast amount of literature is dedicated to the enumerative problem of determining $g_n := \lvert \Aux_n \rvert$, and sometimes also $g_{n,N}:= \lvert\Aux_{n,N}\rvert$, under various general assumptions or for specific examples such as integer partitions, plane partitions or unlabelled \mbox{(un-)rooted} forests, see e.g.~\cite{Granovsky2015,Granovsky2008,Granovsky2006,Palmer1979,Meinardus1954,Hardy1918} for $g_n$ and \cite{Hwang1997,Knessl1990} for $g_{n,N}$. Note that determining $g_{n}$ and $g_{n,N}$ 
is directly related to the limiting distribution and local limit theorems for the number of components in $\mathsf{G}_n$, for example investigated in \cite{Mutafchiev2011,Hwang2001,Bell2000,Erdoes1941}. Another closely related topic that has received a lot of attention is devoted to finding the asymptotic behaviour of the global shape of $\mathsf{G}_n$ and $\mathsf{G}_{n,N}$ in terms of phenomena like condensation or gelation, cf.~\cite{Stufler2020,Mutafchiev2013,Barbour2005,Arratia2003,Mutafchiev1998,Erdoes1941}. Section \ref{sec:other_related_work} highlights some of these results in more detail and makes the connection to this work explicit.

We associate to $\C$ and $\Aux$ the (ordinary) generating series in two formal variables $x$ and $y$
\[
	C(x)
	:= \sum_{k\in\Nat}\lvert\C_k\rvert x^k
	\quad
	\text{ and }
	\quad
	G(x,y)
	:= \sum_{k,\ell \in\Nat} \lvert\Aux_{k,\ell}\rvert x^k y^\ell,
\]
and we use the standard notation $g_{n,N} = \lvert\Aux_{n,N}\rvert = [x^ny^N]G(x,y)$ for all $n,N\in\Nat$. These two power series are known to fulfil the fundamental relation, see for example \cite{Flajolet2009,Leroux1998},
\begin{align}
	\label{eq:ogf_G_wo_weights}
	G(x,y)
	= \e{\sum_{j\ge 1}y^j\frac{C(x^j)}{j}}.
\end{align}
In this paper we consider the prominent and broad case in which the counting sequence $(c_n)_{n \in\mathbb{N}}$ is \emph{subexponential}, and our aim is to study the class $\Aux_{n,N}$ -- what is $g_{n,N}$, how do typical objects look like? -- as $n \to \infty$ \emph{and for all $1 \le N \le n$}. Subexponential sequences appear naturally in combinatorial contexts, the main reason being the presence of square-root singularities in the analysis of  associated generating functions. Here is an example for a prototypical application.

\paragraph{Example.} \emph{Let $\T$ be the class of unlabelled trees, that is, isomorphism classes of connected and acyclic graphs. Then $\mathcal{F} = \scs{Mset}(\T)$ is the class of unlabelled forests. Moreover, see~\cite{Otter1948}, the number of unlabelled trees satisfies \[|\T_n| \sim c \cdot n^{-5/2} \cdot \rho^{-n}\] for some $c > 0$ and $0 < \rho < 1$. What can we say about $\mathcal{F}_{n,N}$?}

~\\
Similar counting sequences, in particular with a polynomial term  $n^{-\alpha}$ for some $\alpha > 1$, appear in a variety of contexts in graph enumeration; so-called subcritical graph classes~\cite{Drmota2011} that include trees, outerplanar and series-parallel graphs are prominent examples. All these counting sequences -- and many more -- are \emph{subexponential}. Let us proceed with a formal definition. In order to do so, we step back from our combinatorial setting and let $(c_k)_{k\in\Nat}$ be a real-valued non-negative sequence. Then we say that $C(x) = \sum_{k\ge 1}c_kx^k$, or $(c_k)_{k\ge 1}$ respectively, is \emph{subexponential} with radius of convergence $\rho>0$, if
\[
	\frac{c_{n-1}}{c_n}
	\sim \rho
	\quad
	\text{and}
	\quad
	c_n^{-1}\sum_{1\le k\le n}c_kc_{n-k}
	\sim 2C(\rho) <\infty,
	\quad
	n\to\infty.
\]
Important examples for subexponential sequences are of the form $c_n \sim \lambda(n) \cdot n^{-\alpha} \cdot \rho^{-n}$ for $\alpha > 1$ and $\lambda(n)$ any slowly varying function, see \cite{Embrechts1984}.

Let us now return to the question investigated in this paper. Given a subexponential $C(x)$ we want to study for $n\in\mathbb{N}$ and all $1\le N \le n$ the number $g_{n,N} = [x^ny^N]G(x,y)$, and moreover, if $C(x)$ is the generating series of some combinatorial class (that is, $(c_k)_{k\in \mathbb{N}}$ is an integer sequence), typical properties of the random multiset $\mathsf{G}_{n,N}$.
A directly related result in this context is~\citep{Bell2000}, where the authors show that the (limiting) distribution of the number of components in a random $\cal C$-multiset $\mathsf{G}_{n}$ is given by a weighted sum of independent Poisson random variables. Equivalently, this means that $g_{n,N}$ can be determined asymptotically for \emph{fixed} of $N$ as $n\to\infty$; thus the enumeration problem is well understood for a bounded number of components. On the other end of the spectrum, let $m\equiv m(C)\in\Nat$ be such that $c_m>0$ and $c_1=\cdots=c_{m-1}=0$. When $C(x)$ is the generating function of a combinatorial class, this means that the size of the smallest possible object in $\C$ is $m$, and so, any $\cal C$-multiset in ${\cal G}_n$ has at most $ n/m$ components. In particular, $n-mN \ge 0$, and if $n -mN = O(1)$ then the structure  of \emph{any}  $\cal C$-multiset of size $n$ with $N$ components is rather simple: except for a bounded number of components of bounded size, all other components are of the smallest possible size $m$.
Our first main result  adresses the enumeration problem in  all other remaining cases, namely when $N, n-mN \to \infty$.
\begin{theorem}
	\label{thm:1_coeff_G(x,y)}
Suppose that $C(x)$ is subexponential and $0<\rho<1$. Let $m = \min\{k\in \mathbb{N}: c_m > 0\}$. Then, as  $n,N,n-mN\to\infty$,
\begin{equation}
\label{eq:mainResult}
	[x^n y^N] G(x,y)
	\sim A \cdot N^{c_m-1} \cdot
	c_{n-m(N-1)},
\end{equation}
where
\[
	A 
	= \frac1{\Gamma(c_m)}\e{\sum_{j\ge 1}\frac{C(\rho^j)-c_m\rho^{jm}}{j\rho^{jm}}}.
\]
\end{theorem}
The proof is in  Section \ref{subsec:proof_1}. Some discussion and remarks are in place. First, by considering real-valued sequences $(c_k)_{k\in\Nat}$ the formula in Theorem \ref{thm:1_coeff_G(x,y)} gives us the asymptotic behaviour of the coefficients of $G(x,y)$ with a priori no combinatorial interpretation. However, \text{if} the sequence is integer-valued and corresponds to the counting sequence of a combinatorial class $\cal C$, then Theorem \ref{thm:1_coeff_G(x,y)} is an enumeration result: it provides us with the number of $\cal C$-multisets of size $n$ and $N$ components, where $N, n-mN \to \infty$.
Second, in combinatorial applications, note that we \emph{always} have that $\rho <1$, as otherwise the subexponentiality of $C(x)$ would imply that $c_k \to 0$ as $k\to\infty$. That is, the assumption $0<\rho<1$ imposes no restriction in the combinatorial setting.

Let us make a third remark that will pave the way to the following results.
From here on 
we solely consider the combinatorial setting. 
Note the right hand side of~\eqref{eq:mainResult}: this formula establishes an explicit connection between $g_{n,N}$ and $c_{n-m(N-1)}$, that is, we do not need the actual counting sequence of $\cal C$ to make statements about $g_{n,N}$. Moreover, a closer look at this formula reveals an unexpected fact. The number of possible ways to choose a multiset of $N$ objects from $\C_m$ is given by $\binom{c_m + N -1}{N} \sim N^{c_m-1}/\Gamma(c_m)$ (this is just a number partition of $N$ in $c_m$ parts). Hence the right hand side of~\eqref{eq:mainResult} is proportional to the number of possibilities to choose $N$ objects from $\C_m$ and one object from $\C_{n-m(N-1)}$; that is, a ``typical'' object from $\Aux_{n,N}$ should essentially consist of a big component with more or less $n - m(N-1)$ vertices and $N-1$ components of the smallest possible size $m$. This is rather extreme, as the size of a component of an object in $\Aux_{n,N}$ is bounded from above by $n-m(N-1)$.

Our next result formalizes this intuition. For $G=\{(C_1,d_1),\dots,(C_k,d_k)\} \in \Aux$ denote by $\Li(G):=\max_{1\le i\le k}\lvert C_i\rvert$ the size of one of its largest components. We show that except for a term $\mathcal{O}_p(1)$, that is, a quantity that is bounded in probability, the largest component in a uniformly drawn object $\mathsf{G}_{n,N}$ from $\Aux_{n,N}$ has indeed size very close to $n-mN$.
\begin{theorem}
\label{thm:3_size_largest_comp_L_n_N}
Suppose that $C(x)$ is subexponential. Then, as  $n,N,n-mN\to\infty$,
\[
	\Li(\mathsf{G}_{n,N}) = n - mN  + {\cal O}_p(1).
\]
\end{theorem}
The proof can be found in Section \ref{subsec:proof_3}. We call the phenomenon established in Theorem~\ref{thm:3_size_largest_comp_L_n_N} \emph{extreme condensation}: we observe that typically our objects have a giant component that is \emph{essentially as large as possible}; its size is close to the largest possible size $n-m(N-1)$. In particular, virtually all other components are of smallest possible size $m$. We are not aware of any other object with a comparable behaviour, at least not in the analytical $(\rho>0$) setting considered here.\footnote{For example, it is known that a factorial weight sequence induces extreme condensation in the balls-in-boxes model, see \cite[Example 19.36]{Janson2012}. In such situations the respective generating series has radius of convergence 0.} Moreover, this behaviour is surprising for one more reason: if we consider the labelled counterparts of our unlabelled objects, in our running example trees, then the typical structure is well known to undergo various phase transitions (from subcritical to condensation) depending on the number of components, but it never becomes as extreme as observed here. See \cite{Janson2012,Panagiotou2018_2} and Section~\ref{subsec:discussion} for a more detailed discussion.

Our final main result addresses the last remaining bit and describes the shape of a typical object from $\mathsf{G}_{n,N}$ when we remove a component of largest size and all components of the smallest possible size $m$. This \emph{remainder} is a multiset of stochastically bounded size and number of components, and we determine the limiting distribution.
To formulate our statement we need some additional notation. Define the class $\C_{>m} = \bigcup_{k>m}\C_k$ equipped with the modified size function $\lvert C\rvert_{>m} := \lvert C\rvert - m$ for $C\in\C_{>m}$. The associated generating function $C_{>m}(x)$ thus equals $(C(x)-c_mx^m)/x^m$; here subtracting $c_mx^m$ accounts for the fact that we remove objects of (the smallest) size $m$ and dividing through $x^m$ all objects in $\C_k$, $k>m$, are treated as objects with size $k-m$. Similar to the formula in \eqref{eq:ogf_G_wo_weights} (setting $y=1$) the class of all multisets $\Aux_{>m} := \scs{Mset}(\C_{>m})$ therefore has generating series
\[
	G_{>m}(x)
	:= \e{\sum_{j\ge 1}\frac{C(x^j)-c_mx^j}{jx^{jm}}}.
\]
Further, the size of an object $G$ in $\Aux_{>m}$ is given by $\lvert G\rvert_{>m} := \lvert G\rvert - m\kappa(G)$. As the coefficients of $C_{>m}(x)$ are given by $(c_{k+m})_{k\in\Nat}$ we deduce that $C_{>m}(x)$ is also subexponential with radius of convergence $\rho$ and $G_{>m}(\rho)<\infty$. 
Define a random variable $\Gamma G_{>m}(\rho)$ on $\Aux_{>m}$ specified by
\begin{equation}
	\pr{\Gamma G_{>m}(\rho) = G}
	= \frac{\rho^{\lvert G\rvert_{>m}}}{G_{>m}(\rho)}
	= \e{-\sum_{j\ge 1}\frac{C(\rho^j)-c_m\rho^j}{j\rho^{jm}}}
	\rho^{\lvert G\rvert - m\kappa(G)},
	\quad 
	G\in\Aux_{>m}.
	\label{eq:boltzmannG>m}
\end{equation}
We remark in passing that this is the -- well-known -- Boltzmann distribution (on ${\cal G}_{>m}$) about which we will talk later extensively. For $G\in\Aux$ let the remainder $\R(G)$ be the multiset obtained after removing all tuples $(C,d)\in G$ with $C\in\C_m$ and one object of largest size from $G$ (this choice can be done in a canonical way by numbering all elements in $\C$). That means, if the object of largest size has multiplicity $d>1$ replace $d$ by $d-1$, otherwise remove the object and its multiplicity $1$ completely from the set. Then the distribution in \eqref{eq:boltzmannG>m} is the limit of the remainder $\R(\mathsf{G}_{n,N})$, see Section \ref{subsec:proof_4} for the proof.
\begin{theorem}
\label{thm:4_distribution_of_remainder_R_n_N}
Suppose that $C(x)$ is subexponential. Then, as  $n,N,n-mN\to\infty$, in distribution 
$
	\R(\mathsf{G}_{n,N})
	\to
	\Gamma G_{>m}(\rho).
$
\end{theorem}
We close this introduction and the presentation of the main results by catching up with our previous example regarding the class $\T$ of unlabelled trees and ${\cal F} = \scs{Mset}({\T})$ the unlabelled forests.

\paragraph{Example (continued)} \emph{Theorem~\ref{thm:1_coeff_G(x,y)} is directly applicable to the class of unlabelled trees. We readily obtain that the number of unlabelled forests of size $n$ with $N$ components satisfies
\[
	f_{n,N} \sim A \cdot\lvert\T_{n-N+1}\rvert \sim A' \cdot (n-N)^{-5/2} \rho^{-n+N},
	\quad \text{for } n,N,n-N \to \infty
\]
and for some constants $A,A' > 0$. Moreover, for this range of $N$, we obtain that with high probability, a random unlabelled forest contains a huge tree with $n - N+ \bigO{1}$ vertices, and $N + \bigO{1}$ ``trivial'' trees that consist of a single vertex.
This is in stark contrast to the known behaviour of random labelled forests, see Section \ref{subsec:discussion} for a detailed discussion, but also from unlabelled models such as random unrooted ordered forests, cf.~\cite{Bernikovich2011}.
} ~\\

We proceed with an application of our results to Benjamini-Schramm convergence of unlabelled graphs with many components. The Benjamini-Schramm limit of a sequence of graphs describes what a uniformly at random chosen vertex typically sees in its neighbourhood and is a special instance of local weak convergence, see also \cite{Aldous2004,Benjamini2001}. 
Given a graph $G=(V,E)$ we form the \emph{rooted} graph $(G,o)$ by distinguishing a vertex $o\in V$. Let $\B$ be the collection of all these rooted graphs.
Then two graphs $(G,o)$ and $(G',o')$ in $\B$ are called isomorphic, $(G,o)\simeq(G',o')$, if there exists an edge-preserving bijection $\Phi$ on the vertex sets of $G$ and $G'$ such that $\Phi(o)=o'$. Hence, the collection $\B_*=\B/{\simeq}$ of equivalence classes in $\B$ under the relation $\simeq$ contains all unlabelled rooted graphs. 
%

Set $B_k(G,o)$ to be the induced subgraph of $(G,o)\in\B_*$ containing all vertices within graph distance $k$ from the root $o$.
Then we say that a sequence of (labelled or unlabelled) simple connected locally finite graphs $(\mathsf{G}_n)_{n\ge 1}$ (possibly random) converges in the \emph{Benjamini-Schramm} (BS) sense to a limiting object $(\mathbbmss{G},\mathbbmss{o})\in\B_*$ if for a vertex $o_n$ being selected uniformly at random from $\mathsf{G}_n$  
\begin{align}
\label{eq:BS_def}
	\lim_{n\to\infty}\pr{B_k(\mathsf{G}_n,o_n)\simeq(G,o)}
	=\pr{B_k(\mathbbmss{G},\mathbbmss{o})\simeq(G,o)},
	\quad k\in\Nat, (G,o)\in\B_*.
\end{align}
Back to our setting, we consider $\C$ to be a class of unlabelled finite connected graphs (with subexponential counting sequence and $m\in\Nat$ denotes the size of the smallest possible graph in $\C$) such that $\Aux=\scs{Mset}(\C)$ is the class of unlabelled graphs with connected components in $\C$. Let as before $\mathsf{G}_{n,N}$ be drawn uniformly at random from $\Aux_{n,N}$. In order to adapt to the setting above we let $(\mathsf{G}_{n,N},o_n)$ denote the connected component around a uniformly at random chosen root $o_n$ in $\mathsf{G}_{n,N}$. Let $\mathsf{C}_n$ be drawn uniformly at random from $\C_n:=\{C\in\C:\lvert C\rvert = n\}$. With this at hand, the extension of BS convergence to non-connected graphs is evident and we obtain the following result.
\begin{proposition}
\label{prop:BS_limit}
Suppose that $C(x)$ is subexponential. Assume that $mN/n \to \lambda \in [0,1)$ as $n,N\to\infty$. If the sequence $(\mathsf{C}_n)_{n\ge 1}$ converges to a limit object $(\mathbbmss{C},\mathbbmss{o})$ in the BS sense, then $\mathsf{G}_{n,N}$ converges as $n,N\to\infty$ to a limit object $(\mathbbmss{G},\mathbbmss{o})$ in the BS sense given by the law
\[
	(1-\lambda) \delta_{(\mathbbmss{C},\mathbbmss{o})}
	+ \lambda \delta_{(C_m,o_m)},
\]	
where $o_m$ is a vertex chosen uniformly at random among the $m$ vertices in $\mathsf{C}_m$. In particular, if $N = o(n)$ we have that $(\mathbbmss{G},\mathbbmss{o})=(\mathbbmss{C},\mathbbmss{o})$.
\end{proposition}
The proof is found in Section \ref{subsec:proof_BS}.
The authors of \cite{Georgalopoulos2016} show that any subcritical class $\C$ of connected unlabelled graphs fulfils the conditions of Proposition \ref{prop:BS_limit}. In the subcritical setting the BS limit of connected unlabelled rooted graphs is also the BS limit of the respective unrooted graphs as shown in \cite{Stufler2017}.
In particular, prior to these works it was shown in \cite{Stufler2019,Stufler2018_2} that the BS limits of unlabelled unrooted trees and of unlabelled rooted trees, also called P{\'o}lya trees, both exist and coincide. Additionally, this limit, say $(\mathbbmss{T},\mathbbmss{o})$, is made explicit in these publications.
\paragraph{Example (further continued)} \emph{
We obtain with Proposition \ref{prop:BS_limit} that the BS limit $(\mathbbmss{F},\mathbbmss{o})$ of $\mathsf{F}_{n,N}$ drawn uniformly at random from all unlabelled forests of size $n$ and being composed of $N$ trees has law, assuming that $N/n\to\lambda\in[0,1)$ as $n,N\to\infty$, 
\[
	(1-\lambda)\delta_{(\mathbbmss{T},\mathbbmss{o})}
	+ \lambda \delta_{\X},
\]
where $\X$ is a single rooted vertex. In other words, with probability $1-\lambda$ the neighbourhood of a uniformly at random chosen vertex from $\mathsf{F}_{n,N}$ looks like the infinite tree $\mathbbmss{T}$ and with probability $\lambda$ the neighbourhood is empty.
}

\paragraph{Proof Strategy} The main idea in the proof is to consider a randomized algorithm/stochastic process that generates $\cal C$-multisets. As it turns out, such an algorithm that outputs elements from $\Aux$ (with a priori \emph{no control} on the size or the number of components!) can be designed by defining the so-called Boltzmann distribution on $\Aux$, see Section~\ref{sec:setup_and_notation} for all details. The crucial property of this algorithm is that all choices it makes are \emph{independent}. Our first contribution is to establish explicitly the connection between the choices of the algorithm and its output; hence the probability that the output is in $\Aux_{n,N}$ can be linked to an event regarding the actual choices of the algorithm. Our second and main contribution is to actually compute the probability of this event; as we will see, this is not at all an easy task, since the involved random variables are not identically distributed and interfere in a complex way with the parameters of the generated object.

In contrast to this work, most proofs in the literature about enumeration of multisets are either conducted from a purely analytical generating function perspective or involve somehow the \emph{conditioning relation} representing the (heavily dependent) number of component frequencies in $\mathsf{G}_n$ of a particular size by independent random variables with negative binomial distributions. This is related to the alternative (and equivalent) representation of the generating function for $\cal C$-multisets~\eqref{eq:ogf_G_wo_weights} given by
\begin{equation}
\label{eq:alternativeG}
	G(x,y) = \prod_{k \ge 1} (1-yx^k)^{-c_k}.
\end{equation}
As opposed to analysing the component spectrum in \cite{Mutafchiev2011,Granovsky2006,Barbour2005,Arratia2003}, our arguments are in the style of \cite{Stufler2020,Panagiotou2018,Stufler2018,Stufler2018_2,Panagiotou2012}:  we use the P{\'o}lya-Boltzmann model representing $\mathsf{G}_n$ as random $\C$-objects attached to cycles of a random permutation, which is helpful to get rid of cumbersome appearances of symmetries and which gives rise to Poisson distributions instead of negative binomials; this difference is reflected by the two different representations~\eqref{eq:ogf_G_wo_weights} and~\eqref{eq:alternativeG}. Then we show that the size associated to fixpoints is dominant and the subexponentiality-feature often referred to as ``single big jump'' guarantees that in fact only one object receives the entire possible size.

\paragraph{Plan of the Paper}
Subsequently, we embed our results in the corpus of existing literature in Section \ref{sec:other_related_work}.
In Section \ref{subsec:discussion} we compare the labelled and the unlabelled setting in light of our results. Then we collect and prove some results about subexponential power series tailored to our needs in Section \ref{sec:subexp}. In Section \ref{sec:proofs} all proofs are presented, where each of the main results is treated in an extra subsection, such that Theorem \ref{thm:1_coeff_G(x,y)} is proven in Section \ref{subsec:proof_1}, Theorem \ref{thm:3_size_largest_comp_L_n_N} in Section \ref{subsec:proof_3} and Theorem \ref{thm:4_distribution_of_remainder_R_n_N} in Section~\ref{subsec:proof_4}.

\paragraph{Notation}
We shall use the following (standard) notation. Given two real-valued sequences $(a_k)_{k\in\Nat}$ and $(b_k)_{k\in\Nat}$ with $b_k \neq 0$ for all $k \ge k_0$ for some $k_0\in\Nat$, we write, as $n\to\infty$,
\begin{itemize}
\item[i)] $a_n \sim b_n$ (``$a_n$ is asymptotically equal to $b_n$'') if $\lim_{n\to\infty}a_n/b_n = 1$,
\item[ii)] $a_n \propto b_n$ (``$a_n$ is asymptotically proportional to $b_n$'') if there exist $0<A_1\le A_2$ such that 
\[
	A_1 
	\le \liminf_{n\to\infty}\bigl\lvert\frac{a_n}{b_n}\bigr\rvert 
	\le \limsup_{n\to\infty}\bigl\lvert\frac{a_n}{b_n}\bigr\rvert 
	\le A_2,
\]
\item[iii)] $a_n = o(b_n)$ if $\lim_{n\to\infty}a_n/b_n = 0$.
\end{itemize}
For a sequence of real-valued random variables $(X_k)_{k\in\Nat}$ and a non-negative sequence $(a_k)_{k\in\Nat}$ we write $X_n=\mathcal{O}_p(a_n)$ (``$X_n$ is stochastically bounded by $a_n$'') if for all $\eps>0$ there exists $K>0$ such that $\limsup_{n\to\infty}\pr{\lvert X_n\rvert \ge K a_n} \le \eps$. In the case $a_k \equiv 1$ we simply say ``$X_n$ is stochastically bounded''.

We  will use the following notation for formal power series. For a $k$-dimensional vector of formal variables $\mathbf{x} = (x_1,\dots,x_k)$ and $\mathbf{d}=(d_1,\dots,d_k)\in\Nat_0^k$ we write $\mathbf{x}^{\mathbf{d}}$ for the monomial $x^{d_1}_1\cdots x^{d_k}_k$. A multivariate power series with real-valued coefficients is given by $A(\mathbf{x}) = \sum_{\mathbf{d}\in\Nat_0^k}a_{\mathbf{d}}\mathbf{x}^{\mathbf{d}}$, where the $a_{\mathbf{d}}$'s are in $\mathbb{R}$. For $\mathbf{d}\in\Nat_0^k$ we write $[\mathbf{x}^{\mathbf{d}}]A(\mathbf{x}) = a_{\mathbf{d}}$ for the coefficient of $\mathbf{x}^{\mathbf{d}}$.

\subsection{(More) Related Work}
\label{sec:other_related_work}
In this section we put our results in the broader context of (asymptotic) enumeration of multisets/weighted integer partitions. The most prominent assumption is that the counting sequence of $\C$ fulfils $c_n \sim \lambda(n)\cdot n^{-\alpha}\cdot \rho^{-n}$ as $n\to\infty$ 
for some slowly varying function $\lambda(\cdot)$ and parameters $\alpha\in\Real, 0<\rho\le 1$. Then there emerge three cases depending on the parameter $\alpha$ determining the behaviour of $C(x)=\sum_{k\ge 1}c_kx^k$ at or near its radius of convergence $\rho$, each giving rise to a fundamentally different picture.

In the \emph{expansive} case $\alpha<1$, which in particular includes the classical and prominent setting $c_k \equiv 1$ of integer partitions, the number of $\cal C$-multisets $g_n = |{\cal G}_n|$ is well-understood \cite{Granovsky2006}. However, general results about the uniformly drawn element $\mathsf{G}_n$ -- in particular the distribution of the number of components that is of interest here -- are not known without any extra conditions. For example, under the assumption of \emph{Meinardus} scheme of conditions, a set of analytic assumptions that in particular imply that $\rho=1$, the number of components of $\mathsf{G}_n$ fulfils various local limit theorems, see \cite{Mutafchiev2011}.
The authors of \cite{Granovsky2008} state that expansive multisets with counting sequence $c_k=Ck^{-\alpha}$ for some $C>0$ fulfil Meinardus conditions. For example, integer partitions $(c_k=1)$ and plane partitions $(c_k=k)$ are encapsulated by this approach. Hence, they call the Meinardus case \emph{quasi-expansive}. 
For quasi-expansive sequences it is established in~\cite{Mutafchiev2013} that the size of the largest component of $\mathsf{G}_n$ is with high probability of size $\Theta(1)\cdot n^{1/(2-\alpha)}\log n$. 
The broader picture here is that the number of components in $\mathsf{G}_n$ is typically unbounded and the size of the largest component is sublinear in $n$. On the other hand, much less is known for the number of $\cal C$-multisets with $N$ components $g_{n,N} = |{\cal G}_{n,N}|$ and the typical shape of $\mathsf{G}_{n,N}$ for arbitrary $1 \le N \le n$. Nevertheless, there is one important exception, namely the case of integer partitions. There, depending on whether $N$ is $\mathcal{O}(n^{1/2})$ or $\omega(n^{1/2})$ the asymptotic behaviour of the number of partitions of $n$ into $N$ parts is given by different formulas, see \cite{Knessl1990}. For $N \ge (1+\eps)/(\sqrt{2/3}\pi)\sqrt{n}\log n$ for any $0<\eps<1$ it is even true that $g_{n,N}\sim g_{n-N}$ \cite{Hwang1997}. In general, it is reasonable to conjecture that the shape of $\mathsf{G}_{n,N}$ depends on the asymptotic regime of $N$; this, however, is a topic for a completely different paper. 

The \emph{logarithmic} case, where $\alpha=1$ and $\lambda\equiv\lambda(\cdot)$ constant, is concomitant with similar effects. The number of components in $\mathsf{G}_n$ is typically of order $\lambda \log n$ \cite[Thm.~8.21]{Arratia2003} and, denoting by $(L_1,\dots,L_N)$ the $N$ largest component sizes of $\mathsf{G}_n$, then $n^{-1}(L_1,\dots,L_N)$ has a limiting distribution that is Poisson-Dirichlet \cite[Thm.~6.8]{Arratia2003} implying that $\mathsf{G}_n$ is composed of several ``large'' objects. The same is  true for $\mathsf{G}_{n,N}$ with $N\in\Nat$ fixed, where the $N-1$ smallest components have with high probability sizes $(n^{U_i})_{1\le i\le N-1}$ for iid uniformly distributed random variables $(U_i)_{1\le i\le N-1}$ \cite[Thm.~6.9]{Arratia2003}.
To our knowledge, there are no results (regarding asymptotic enumeration or structural properties) for all other $N$.

In contrast to all previous cases, the phenomenon of condensation is observed in the \emph{convergent} case, where $\alpha>1$: the single largest component of $\mathsf{G}_n$ is of size $n-\mathcal{O}_p(1)$ and its number of components converges in distribution, see \cite{Barbour2005}. In accordance with the results observed for the convergent case the number of components in $\mathsf{G}_n$ in the subexponential setting has a limiting distribution given by a weighted sum of independent Poisson random variables \cite{Bell2000}. Equivalently, this means that $g_{n,N}$ can be determined asymptotically for \emph{fixed} values of $N\in\Nat$. As for the global shape of $\mathsf{G}_n$, the results in \cite{Stufler2020} imply that the remainder obtained after removing the largest component converges in distribution to a limit given by the so-called P{\'o}lya-Boltzmann distribution, see~\cite{Bodirsky2011}. 

\subsection{Discussion - The Unlabelled vs.~the Labelled Setting}
	\label{subsec:discussion}
In what follows we will have a closer look at the resemblances and surprising disparities between multisets, which are typically associated to \emph{unlabelled} structures, and sets of \emph{labelled} combinatorial structures. We refer the reader to the books \cite{Flajolet2009,Leroux1998} for an excellent exposition to combinatorial classes. Another vast source of references and examples is the tour-de-force paper \cite{Janson2012} entailing many results about the \emph{balls-in-boxes} model (Section 11), which by choosing the weight sequence $(c_k/k!)_{k\in\Nat}$ implies the labelled set-construction.

Given a labelled class $\C^\iota$  we may form the analogon to the multiset construction discussed in this work. Initially we pick $C_1,\dots,C_k$ from $\C^{\iota}$ and let $n$ be the total size. Then we partition $\{1,\dots,n\}$ into sets $L_1,\dots,L_k$ such that $L_i=\lvert C_i\rvert$ for all $i$. Subsequently, we canonically assign the labels in $L_i$ to $C_i$ for all $i$. The outcome of this procedure is a labelled set, where each of the labels in $\{1,\dots,n\}$ appears exactly once. The notion of size and number of components carries over from the multiset construction. Let us call the collection of all such labelled sets $\Aux^\iota = \scs{Set}(\C^\iota)$ and introduce the sets $\Aux^\iota_n$ and $\Aux^\iota_{n,N}$ of objects in $\Aux^\iota$ of size $n$ and of size $n$ having $N$ components, respectively. Further, let $c^\iota_k := \lvert \{C\in\C^\iota:\lvert C\rvert=k\}\rvert$ for $k\in\Nat$.
Similarly to \eqref{eq:ogf_G_wo_weights} the bivariate (exponential) generating series related to this case is known (\cite{Flajolet2009,Leroux1998}) to be
\[
	G^\iota(x,y)
	= \e{ y C^\iota(x) },
	\quad
	\text{where  }
	C^\iota(x)
	= \sum_{k\ge 1}c_k^\iota \frac{x^k}{k!}.
\]
In complete analogy to the unlabelled case we are here intrerested in the number $g^\iota_{n,N} = \lvert\Aux^\iota_{n,N}\rvert$ of sets of size $n$ and $N$ components, and for properties of the uniform random elements $\mathsf{G}^\iota_n$ and $\mathsf{G}^\iota_{n,N}$ from ${\cal G}^\iota_n$ and ${\cal G}^\iota_{n,N}$, respectively.

First of all, the case $N = \bigO{1}$ is treated (together with the unlabelled case) in~\cite{Bell2000}. There it is shown that in the subexponential setting both $\kappa(\mathsf{G}_n)$ and $\kappa(\mathsf{G}^\iota_n)$ converge in distribution; hence, for a fixed number of components the labelled and unlabelled cases behave qualitatively the same. Also the global structure of the associated random variables $\mathsf{G}_n$ and $\mathsf{G}^\iota_n$ is in both cases governed by the same condensation effect, see~\cite{Stufler2018,Stufler2020}.
However, the situation changes as $N \to \infty$.
The works \cite{Panagiotou2018_2,Janson2012} treat this topic extensively\footnote{In particular we want to highlight \cite[Theorems 18.12, 18.14, 19.34, 19.49]{Janson2012}.}: under the condition that $c_n^\iota\sim bn^{-(1+\alpha)}\rho^{-n}n!$ for $b>0$ and $\alpha>1$ as $n\to\infty$ there emerges a ``trichotomy" ($1<\alpha\le 2$) and in some cases a ``dichotomy" ($\alpha>2$) depending on the asymptotic regime of $N$. 
To illustrate the nature of these results, let us consider the class of labelled  trees $\T^\iota$ such that $\F^\iota = \scs{Set}(\T^\iota)$ is the class of labelled forests. The well-known formula by Cayley states that $t^\iota_n  = n^{n-2} \sim (2\pi)^{-1/2}n^{-5/2} \eul^n n!$, so that $\alpha = 3/2$. Abbreviating by $f^\iota_{n,N}$ the number of forests on $n$ nodes and $N$ trees, the following detailed result exposing two phase transitions is known. Let $N:=\floor{\lambda n}$, then
\begin{align*}
	\frac{N!}{n!} \cdot f^{\iota}_{n,N} \sim
	\begin{cases}
		c_{-}(\lambda) n^{-3/2} \eul^n 2^{-N}, & \lambda \in (0,1/2) \\
		c n^{-2/3} \eul^n 2^{-N}, & \lambda = 1/2 \\
		c_{+}(\lambda) n^{-1/2} f(\lambda)^n, &\lambda\in(1/2,1) 
	\end{cases},
	\quad
	n\to\infty,
\end{align*}
for positive real-valued continuous functions $c_{-/+}(\lambda), f(\lambda)$ and a constant $c$; note that the critical exponent jumps form $3/2$ to $2/3$ and then to $1/2$. All in all, the main results of the present paper reveal substantial differences between the labelled and the unlabelled case already at the level of the counting sequences: as we stated before in our example, the number of unlabelled forests of size $n$ with $N = \floor{\lambda n}$ components is asymptotically equal to $A\cdot (1-\lambda)^{-5/2} n^{-5/2} \rho^{n-N}$; in particular, the critical exponent does not vary. 

The aforementioned variation in the critical exponent has also important consequences for the global structure of a  labelled forest $\mathsf{F}^\iota_{n,N}$ drawn uniformly at random from the set of labelled forests of size $n$ composed of $N$ trees.  Three different cases emerge as $n$ approaches infinity: 
\begin{enumerate}
\item In the case where there are ``few'' components ($0 < \lambda < 1/2$), most of the mass is concentrated in one large tree containing a linear fraction (that is  $1 - 2\lambda$) of all nodes and the remaining $N-1$ trees all have size $\mathcal{O}_p(n^{2/3})$.
\item In the case where the ratio between components and total size is ``balanced'' ($\lambda = 1/2$) all trees have size $\mathcal{O}_p(n^{2/3})$.
\item Whenever there are ``many'' components with respect to the total number of nodes ($1/2 < \lambda < 1$), all trees are small in the sense that their size is stochastically bounded by $\log n$. 
\end{enumerate}
For a detailed discussion for what happens near the critical point $\lambda=1/2$ see also \cite{Tomasz1992}. This is again substiantially different to the unlabelled case, where we showed that extreme condensation dominates the picture for all values of $N$.

\section{Subexponential Power Series}
	\label{sec:subexp}
In this section we collect (and prove) some properties of subexponential power series that will be quite handy in the rest of paper. Many of the definitions and statements shown here are taken from \citet*{Embrechts1984} or \citet*{Foss2013} and adapted to the discrete case, see also~\citet{Stufler2020}.
\begin{definition}
\label{def:subexp}
A power series $C(x)=\sum_{k \ge 0} c_k x^k$ with non-negative coefficients and radius of convergence $0 < \rho < \infty$ is called \emph{subexponential} if 
\begin{align}
	\frac{1}{c_n} \sum_{0\le k \le n} c_{n-k} c_k &\sim 2 C(\rho) < \infty \tag{$S_1$}\label{eq:subexp1}
	\quad\text{and} \\
	\frac{c_{n-1}}{c_n} &\sim \rho, \quad n\to\infty. \tag{$S_2$}
\label{eq:subexp2}
\end{align}
\end{definition}
Note that the radius of convergence of a power series $C(x)$ satisfying~\eqref{eq:subexp2} (in particular of any subexponential power series) is $\rho$ and that eventually $[x^n]C(x) > 0$, where as usual, $[x^n]C(x) = c_n$ denotes the coefficient of $x^n$ in $C(x)$.
Any arbitrary subexponential power series $C(x)$ with radius of convergence $\rho$ induces the probability generating series of a $\Nat_0$-valued random variable by setting
\[
	d_k := \frac{c_k \rho^k}{C(\rho)}, \quad n\in\mathbb{N}_0.
\]
Then $D(x) = \sum_{k \ge 0} d_k x^k$ is subexponential with $\rho=1$ and $D(\rho) = 1$. There are several results about the asymptotic behaviour of sums of random variables with such a subexponential generating series. Here we will need Lemma \ref{lem:subexp_several_prop} \ref{lem:random_stopped_sum} below, which corresponds to determining the probability that a randomly stopped sum of random variables with distribution $(d_k)_{k\ge 0}$ attains a large value.
Moreover, Lemma \ref{lem:subexp_several_prop} \ref{lem:long_sum_exp} will be particularly useful, since it provides bounds holding uniformly in the given parameters. 
In Lemma \ref{lem:subexp_several_prop} \ref{lem:big_jump_principle} we present and prove a statement often referred to -- with various interpretations -- as ``principle of a single big jump''. The dominant contribution to a large sum of subexponential random variables stems typically from one single summand.
\begin{lemma}
\label{lem:subexp_several_prop}
Let $(D_i)_{i \in \Nat}$ be iid $\Nat_0$-valued random variables with probability generating function $D(x)$. Assume that $D(x)$ is subexponential with radius of convergence $1$. For $p\in\Nat$ let $S_p := \sum_{1\le i \le p} D_i$ and $M_p := \max\{D_1,\dots,D_p\}$. Then the following statements are true. 
\begin{enumerate}[label=(\roman*)]
\item 
\label{lem:random_stopped_sum} \cite[Theorem 4.30]{Foss2013}
Let $\tau$ be a $\Nat_0$-valued random variable independent of $(D_i)_{i\in\Nat}$. Further, assume that the probability generating function of $\tau$ is analytic at $1$. Then
\[
	\pr{S_\tau = n} \sim \ex{\tau} \pr{D_1 = n},
	\quad n\to\infty.
\]
\item
\label{lem:long_sum_exp} \cite[Theorem 4.11]{Foss2013}
For every $\delta > 0$ there exists $n_0 \in \Nat$ and a  $C>0$ such that
\[
	\pr{S_p = n}
	\le C (1 + \delta)^p \pr{D_1 = n},
	\quad
	\text{for all } n\ge n_0,~ p\in\Nat.
\]
\item
\label{lem:big_jump_principle}
For any $p\ge 2$
\[
	( M_p \mid S_p = n )
	= n + \mathcal{O}_p(1), 
	\quad n\to\infty.
\]
\end{enumerate}
\end{lemma}
\begin{proof}[Proof of Lemma \ref{lem:subexp_several_prop} \ref{lem:big_jump_principle}] 
Let $\eps>0$ be arbitrary. To prove the claim we will establish the existence of $K\in\Nat$ such that
\[
	\lim_{n\to\infty} \pr{\lvert M_p-n\rvert \ge K\mid S_p = n} 
	< \eps.
\]
Clearly under the condition $S_p = n$ we have $M_p \le n$. Thus 
\begin{align}
\label{eq:|M-n|geK_bigjump}
	\pr{\lvert M_p - n \rvert \ge K\mid S_p=n}
	= \sum_{k\ge K} \pr{ M_p = n-k \mid S_p=n}.
\end{align}
Since $D_1,\dots,D_p$ are iid we obtain for any $k\ge K$
\[
	\pr{ M_p = n-k \mid S_p=n}
	\le  \pr{\bigcup_{1 \le i \le p} \{D_i = n-k\} \mid S_p = n}
	= p \frac{\pr{D_1 = n-k} \pr{S_{p-1}=k}}{\pr{S_p=n}}.
\]
Together with Lemma \ref{lem:subexp_several_prop} \ref{lem:long_sum_exp} we find some constant $C>0$ such that for $k\ge K$ sufficiently large
\[
	\pr{S_{p-1} = k}
	\le C (1+\eps)^{p-1}\pr{D_1=k}.
\]
Part \ref{lem:random_stopped_sum} justifies for $n$ sufficiently large that
$
	\pr{S_p = n}
	\ge (1-\eps) p \pr{D_1 = n}.
$
All in all, for a suitably chosen constant $C(p)$ the expression in \eqref{eq:|M-n|geK_bigjump} can be estimated by 
\[
	\pr{\lvert M_p - n \rvert \ge K\mid S_p=n}
	\le C(p) \sum_{k\ge K} \frac{\pr{D_1 = n-k}\pr{D_1 = k}}{\pr{D_1=n}}.
\]
Due to property \eqref{eq:subexp1} we conclude that this  is smaller than $\eps$ choosing $K$ large enough and the proof is finished.
\end{proof}
The following lemma establishes asymptotics for the coefficients of the product of two power series.
\begin{lemma}{\cite[Thm.~3.42]{Burris2001} or \cite[Ex.~178]{Polya1970}}
\label{lem:coeff_product}
Let $A(x), B(x)$ be power series such that $A$ satisfies~\eqref{eq:subexp2}. Assume that the radii of convergence $\rho_A$ and $\rho_B$ of $A$ and $B$, respectively, satisfy $0 < \rho_A < \rho_B$ and $B(\rho_A) \neq 0$. Then
\[
	[x^n]A(x)B(x) \sim B(\rho_A) \cdot [x^n]A(x), 
	\quad n\to \infty.
\]
\end{lemma}
Note that Lemma \ref{lem:coeff_product} does not require $A$ to be subexponential, neither does it require that $B$ has non-negative coefficients only. We will later apply the lemma with (powers of)
\[
	A(x) 
	:= \frac{1}{1-\rho_A^{-1} x} = \sum_{k \ge 0} (\rho_A^{-1}x)^k
\]
for some $\rho_A>0$. Then $A(x)$ has $\eqref{eq:subexp2}$ with radius of convergence $\rho_A$ but $A(\rho_A)=\infty$; in particular, property $\eqref{eq:subexp1}$ is not satisfied and $A$ is not subexponential.

As a final remark in this section we make the following observation, which we shall use mostly without further reference. Suppose that for two sequences $(a_n)_{n\in\Nat_0}, (b_n)_{n\in\Nat_0}$ in $\mathbb{R}$ we know that $a_n \sim b_n$ as $n \to \infty$. Then the ratio $a_n / b_n$ is bounded unless $b_n = 0$, that is, 
\begin{equation}
\label{eq:asimbuniform}
	a_n \sim b_n
	\implies 
	\text{there is $A > 0$ such that }
	a_n \le A |b_n| \text{ for all $n\in \Nat_0$ such that $b_n \neq 0$}.
\end{equation}

\section{Proofs}
	 \label{sec:proofs}

We briefly (re-)collect all assumptions and fix the notation needed in this section. Note that Theorem \ref{thm:1_coeff_G(x,y)} is valid for real-valued sequences $(c_k)_{k\in\Nat}$ whereas the remaining theorems are only reasonable in a combinatorial setting. We will hence put everything into a combinatorial context and outsource the few modifications needed for the general real-valued setting into Section \ref{sec:remark_real-valued}.

For a combinatorial class $\C$ let $C(x) = \sum_{k \ge 1} c_k x^k$ denote a power series with coefficients $c_k := \lvert\{C\in\C:\lvert C\rvert =k \}\rvert$, $k\in\Nat$. Further, let 
\begin{equation}
\label{eq:mc}
	m = m_C := \min\{ k\in\Nat ~:~ c_k > 0\}
\end{equation}
be the index of the first coefficient that does not equal zero. We also assume that $C(x)$ is subexponential, which implies that the radius of convergence fulfils $0<\rho<1$. However, the subexponentiality feature is only needed in the very last step of the proof, cf.~Lemma~\ref{lem:probability_E_n_P_N_asymptotic_wo_weights}; all other statements preceding this lemma are valid even without this asssumption as long as $0<\rho<1$.
Further we define $G(x,y) := \e{\sum_{j\ge 1} C(x^j)y^j/j}$ and $G(x) := G(x,1)$.
We begin with two auxiliary statements. The first one is about the radius of convergence of $G(x)$.
\begin{lemma}
\label{lem:(roc)Gfinite}
Assume that $C(x)$ is a power series with non-negative real-valued coefficients and radius of convergence $0 < \rho < 1$ and $C(0) = 0$, $C(\rho) < \infty$. Then $G(x)$ has radius of convergence $\rho$ and $G(\rho) < \infty$.
\end{lemma}
\begin{proof}
From the definition of $G$ we obtain that $G(x) = e^{C(x)}H(x)$, where $\log H(x) = \sum_{j \ge 2} C(x^j)/j$. Since $\rho \in (0,1)$ we obtain for any $\eps > 0$ such that $(1+\eps)\rho^2 < \rho$ and any $j \ge 2$
\[
	C\big((1+\eps)\rho^j\big)
	= \sum_{k \ge 1} c_k (1+\eps)^k\rho^{jk}
	= (1+\eps)\rho^j \sum_{k \ge 1} c_k ((1+\eps)\rho^j)^{k-1}
	< 
	 (1+\eps)\rho^{j-1} C(\rho).
\]
In particular, $H((1+\eps)\rho) < \infty$ and the radius of convergence of $H$ is larger than $\rho$. Thus, the radius of convergence of $G$ is $\rho$, and $G(\rho) = e^{C(\rho)}H(\rho)< \infty$.
\end{proof}
The second statement is a purely technical result that we will be handy.
\begin{lemma}{\cite[Theorem VI.1]{Flajolet2009}}
\label{lem:asymptpowergeom}
Let $\alpha,\beta\in\Real_+$. Then
\[
	[x^n] (1-\beta x)^{-\alpha} 
	\sim \frac{n^{\alpha -1}}{\Gamma(\alpha)} \beta^{n},
	\quad
	n\to\infty.
\]
\end{lemma}
In the remaining part of this section we introduce the Boltzmann model as helpful tool to reduce our problems to the investigation of iid random variables in Section \ref{sec:setup_and_notation}. Subsequently, we present the proofs of our three main theorems in Sections \ref{subsec:proof_1}-\ref{subsec:proof_4}. At last we prove Proposition \ref{prop:BS_limit} in Section \ref{subsec:proof_BS}.

\subsection{Setup and Notation}
\label{sec:setup_and_notation}
In this section we will introduce the \emph{Boltzmann model} from the pioneering paper \cite{Duchon2004}, which has found various applications in the study of the typical shape of combinatorial structures, see for example \cite{Drmota2019,Addario2019,Stufler2018,Panagiotou2016,Drmota2014,Curien2014,Bernasconi2010,Panagiotou2010}. With the help of this model we translate the initial problem of extracting coefficients of the multiset ogf of unlabelled classes into a probabilistic question. This gives us the proper idea for the general approach for arbitrary functions of the form \eqref{eq:ogf_G_wo_weights}, i.e. when the coefficients are not necessarily integers.
Further, the formalisation via this model will allow us to prove the extreme condensation phenomenon. 

Assume that $z \in \mathbb{R}_+$ is chosen such that $C(z)>0$ is finite.  The unlabelled Boltzmann model defines a random variable $\Gamma C(z)$ taking values in the entire space $\C$ through
\[
	\pr{ \Gamma C(z) = C} =\frac{z^{\card{C}}}{C(z)},
	\quad
	C\in\C.
\]
In complete analogy the random variable $\Gamma G(z)$ is defined on $\Aux = \scs{Mset}(\C)$ of multisets containing $\C$-objects, where in this case the parameter $z>0$ is such that $G(z) := G(z,1)<\infty$ in \eqref{eq:ogf_G_wo_weights}. In the rest of this section we fix  $z = \rho$ recalling that $0<\rho<1$ is the radius of convergence of $C$. Then, in virtue of Lemma~\ref{lem:(roc)Gfinite}, $G$  has radius of convergence $\rho$ and $G(\rho)< \infty$, so that both $\Gamma C(\rho), \Gamma G(\rho)$ are well-defined, and we just write $\Gamma C , \Gamma G$.

Let $g_n$ be the number of objects of size $n$ in $\Aux$ and $g_{n,N}$ those of size $n$ comprised of $N$ components.
By using Bayes' Theorem and that the Boltzmann model induces a uniform distribution on objects of the same size, we immediately obtain
\begin{equation}
\label{eq:gnNBoltzmann}
	\frac{g_{n,N}}{g_n}
	= \pr{ \kappa (\Gamma G) = N \mid  \card{\Gamma G} = n} 
	= \pr{ \card{\Gamma G} = n \mid \kappa (\Gamma G) = N } 
	\frac{\pr{\kappa(\Gamma G) = N}}{\pr{\card{\Gamma G} = n}},
	~~
	n,N\in\Nat.
\end{equation}
To get a handle on this expression we exploit a powerful description of the distribution of $\Gamma G(z)$ in terms of $\Gamma C(\cdot)$, derived in \cite{Flajolet2007}.
In the next steps, the notation $\bigsqcup_{j\in J} A_j$ is used to denote a multiset of elements $A_j$ from a set $\mathcal{A}$, $j\in J$ being indices in some countable set $J$.
That is, multiple occurrences of identical elements are allowed and $\bigsqcup_{j\in J} A_j$ is completely determined by the different elements it contains and their multiplicities.
\begin{enumerate}[label={(\arabic*)}]
\item Let $(P_j)_{j\ge 1}$ be independent random variables, where $P_j \sim \pois{C(\rho^j)/j}$.
\item Let $(\gamma_{j,i})_{j,i\ge 1}$ be independent random variables with $\gamma_{j,i} \sim \Gamma C(\rho^j)$ for $j,i\ge 1$. 
\item For $j,i\ge 1$ and $1\le k \le j$ set $\gamma_{j,i}^{(k)} = \gamma_{j,i}$, that is, make $j$ copies of $\gamma_{j,i}$. Let $\Lambda G := \bigsqcup_{j \ge 1} \bigsqcup_{1 \le i \le P_j} \bigsqcup_{1 \le k \le j } \gamma_{j,i}^{(k)}$.
\end{enumerate}
Intuitively, we interpret $P_j$ as the number of $j$-cycles in some not further specified permutation and to each cycle of length $j$ we attach $j$ times an identical copy of a $\Gamma C(\rho^j)$-distributed $\C$-object. Afterwards we discard the permutation and the cycles and keep the multiset of the generated $\C$-objects. This construction is also made explicit in \cite[Prop.~37]{Bodirsky2011}. 
\begin{lemma}{\cite[Prop.~2.1]{Flajolet2007}}
	\label{lem:algo_same_distri_boltzmann}
The distributions of $\Gamma G$ and $\Lambda G$ are identical.
\end{lemma}
This statement  paves the way to study $\Gamma G$. In particular, if we write $C_{j,i} = |\gamma_{j,i}|$, note that the definition of $\Lambda G$ guarantees that in distribution
\[
	\kappa(\Gamma G) = \sum_{j\ge 1} j P_j
	\quad
	\text{and}
	\quad
	|\Gamma G| = \sum_{j\ge 1} j \sum_{1 \le i \le P_j} C_{j,i}.
\]
So, let us for $n,N\in \Nat$ define the events
\begin{align}
	\label{eq:def_E_n_P_N}
	\Pa_N := \left\{ \sum_{j\ge 1} j P_j = N \right\}
	\quad
	\text{and}
	\quad
	\E_n := \left\{ \sum_{j\ge 1} j \sum_{1 \le i \le P_j} C_{j,i} = n \right\}.
\end{align}
With $\pr{\E_n} = \pr{\card{\Lambda G} = n} = g_n \rho^n/G(\rho)$ at hand, Lemma \ref{lem:algo_same_distri_boltzmann} and~\eqref{eq:gnNBoltzmann} then guarantee that
\begin{align}
	\label{eq:coeff_in_prob_wo_weights_boltzmann1}
	g_{n,N}
	= G(\rho)\rho^{-n}\pr{\E_n\mid\Pa_N} \pr{\Pa_N}.
\end{align} 
Note that for all $1\le i\le P_j$ and $j\in\Nat$, we have
\begin{align}
	\label{eq:prob_C_j_i=k}
	\pr{C_{j,i} = k} = \frac{c_k \rho^{jk}}{C(\rho^j)},
	\quad
	k\in\Nat.
\end{align}
Equation \eqref{eq:coeff_in_prob_wo_weights_boltzmann1} enables us to reduce the problem of determining $g_{n,N} = [x^n y^N]G(x,y)$ to the problem of determining the probability of the events $\Pa_N$ and $\E_n$ conditioned on $\Pa_N$.
\subsubsection{Remarks to Theorem \ref{thm:1_coeff_G(x,y)}}
\label{sec:remark_real-valued}
In Theorem \ref{thm:1_coeff_G(x,y)} we consider $(c_k)_{k\in\Nat}$ to be a real-valued non-negative sequence and assume $0<\rho<1$. In complete analogy to the discussion prior to this subsection let $P_j\sim \pois{C(\rho^j)/j}$ for $j\in \Nat$, $(C_{j,1},\dots,C_{j,P_j})_{j \in\Nat}$ be as in \eqref{eq:prob_C_j_i=k}, and assume that all these variables are independent.
As a matter of fact, also in this (more general) case we obtain exactly the same representation of $[x^n y^N]G(x,y)$ in terms of $\E_n$ and $\Pa_N$ defined in~\eqref{eq:def_E_n_P_N} without using the combinatorial Boltzmann model.
\begin{lemma}
	\label{lem:coeff_in_prob_wo_weights}
Let $C(x)$ be a power series with non-negative real-valued coefficients and radius of convergence $0<\rho<1$ at which $C(\rho) < \infty$. Then
\[
	[x^n y^N] G(x,y) 
	= G(\rho) \rho^{-n} \pr{\E_n\mid \Pa_N} \pr{\Pa_N},
	\quad
	n,N\in\Nat.
\]
\end{lemma}
\begin{proof}
We begin with the simple observation 
\begin{equation}
\label{pPNen}
\begin{split}
	\pr{\Pa_N, \E_n} 
	&= [x^n y^N] \sum_{k \ge 0} \sum_{\ell \ge 0} \pr{\Pa_k, \E_{\ell}} x^{\ell} y^k  \\
	&= [x^n y^N] \sum_{k \ge 0} y^k \sum_{\sum_{j \ge 1}jp_j=k} 
	\prod_{j \ge 1}\pr{P_j=p_j}
	\sum_{\ell \ge 0} \pr{\sum_{j \ge 1} j \sum_{1 \le i \le p_j} C_{j,i} = \ell} x^{\ell}.
\end{split}
\end{equation}
We will study this expresion by first simplifying the sum over $\ell$, then the sum over all $p_j$'s, and eventually the sum over $k$. We begin with the sum over $\ell$. For a $\Nat_0$-valued random variable $A$ let $A(x) := \sum_{\ell\ge 0}\pr{A_j = \ell}x^\ell$ denote its probability generating series. Then, if $(A_j)_{j \in \Nat}$ is a sequence of independent $\Nat_0$-valued random variables,
\begin{align}
	\label{eq:pgf_convolution}
	(A_1+\dots+A_m)(x)
	= \prod_{1\le j\le m} A_j(x),
	\quad
	m\in\Nat.
\end{align}
Let us write $C_j(x)$ for the probability generating series of $jC_{j,i}$; note that the actual value of $i$ is not important, since the $(C_{j,i})_{i \in \Nat}$ are iid. Then, whenever $\sum_{j\ge 1} p_j$ is finite,~\eqref{eq:pgf_convolution} implies
\[
	\sum_{\ell \ge 0} \pr{\sum_{j \ge 1} j \sum_{1 \le i \le p_j} C_{j,i} = \ell} x^{\ell}
	= \prod_{j\ge 1} C_j(x)^{p_j}.
\]
Noting that $jC_{j,1}$ takes only values in the lattice $j\Nat_0$, we obtain 
\[
	C_j(x)
	= \sum_{\ell \ge 0} \pr{jC_{j,1}=\ell}x^\ell
	= \sum_{\ell \ge 0} \pr{ C_{j,1} = \ell} x^{j \ell}
	= \frac{1}{C(\rho^j)} \sum_{\ell \ge 0} c_\ell \rho^\ell x^{j \ell}
	= \frac{C((\rho x)^j)}{C(\rho^j)}.
\]
We deduce
\[
	\sum_{\ell \ge 0} \pr{\sum_{j \ge 1} j \sum_{1 \le i \le p_j} C_{j,i} = \ell} x^{\ell}
	= \prod_{j\ge 1} \left(\frac{C((\rho x)^j)}{C(\rho^j)}\right)^{p_j}.
\]
This puts the sum over $\ell$ in~\eqref{pPNen} in compact form. To simplify the sum over the $p_j$'s in~\eqref{pPNen} define independent random variables $(H_j)_{j\ge 1}$ with $H_j \sim \pois{C((\rho x)^j)/j}$. Then
\[
	\sum_{\sum_{j \ge 1}jp_j=k} 
	\prod_{j \ge 1}\pr{P_j=p_j}
	\left(\frac{C((\rho x)^j)}{C(\rho^j)}\right)^{p_j}
	=\frac{G(\rho x,1)}{G(\rho,1)}
	\pr{\sum_{j \ge 1} j H_j = k}.
\]
By similar reasoning as before the probability generating function of $jH_j$ is given by
\[
	\sum_{\ell \ge 0} \pr{H_j = \ell} y^{j \ell}
	= \e{- C((\rho x)^j)/j} \sum_{\ell \ge 0} \frac{(C((\rho x)^j)y^j/j)^\ell}{\ell!}
	= \frac{\e{C((\rho x)^j)y^j/j}}{\e{C((\rho x)^j)/j} }.
\]
Applying \eqref{eq:pgf_convolution}, where we set $A_j := jH_j$, in combination with this identity and plugging everything into~\eqref{pPNen} yields
\[
	\sum_{k \ge 0} \pr{\sum_{j \ge 1} jH_j = k} y^k
	= \frac{G(\rho x,y)}{G(\rho)}.
\]
All in all, we have shown that $\pr{\Pa_N, \E_n} = G(\rho)^{-1}  [x^n y^N] G(\rho x ,y)$. With $[x^n]F(ax) = a^n [x^n]F(x)$ for any power series $F$ and $a\in\Real$ we finish the proof.
\end{proof}

\subsection{Proof of Theorem \ref{thm:1_coeff_G(x,y)}}
	\label{subsec:proof_1}

 Let $\Pa_N$, $\E_n$ be as in the previous section, see~\eqref{eq:def_E_n_P_N},
 where $P_j \sim\pois{C(\rho^j)/j}$ and $C_{j,1},\dots,C_{j,P_j}$ for ${j\in\Nat}$ have the distribution specified in~\eqref{eq:prob_C_j_i=k}. Moreover, we assume that all these random variables are independent. Equipped with Lemma~\ref{lem:coeff_in_prob_wo_weights} from the previous section, the proof of Theorem~\ref{thm:1_coeff_G(x,y)} boils down to estimating $\pr{\E_n\mid \Pa_N}$ and $\pr{\Pa_N}$. Before we actually do so, let us introduce some more auxiliary quantities. Set
\[
	P := \sum_{j \ge 1} j P_j
	\quad
	\text{and}
	\quad
	P^{(\ell)} := \sum_{j > \ell} j P_j,
	\quad
	\ell\in\Nat_0.
\]
With this notation, $\Pa_N$ is the same as $\{P = N\}$ and $\{P^{(0)} = N\}$. Moreover, recall~\eqref{eq:mc} and set
\begin{align}
	\label{eq:def_L_R}
	L := \sum_{1 \le i \le P_1} (C_{1,i} - m)
	\quad
	\text{and}
	\quad
	R := \sum_{j \ge 2} j \sum_{1 \le i \le P_j} (C_{j,i} - m).
\end{align}
With this notation
\begin{equation}
\label{eq:EZP}
	\pr{\E_n \mid \Pa_N} = \pr{  L + R = n-mN \mid \Pa_N}.
\end{equation}
The  driving idea behind these definitions is that the random variables $C_{j,i}-m$, for $j \ge 2$, have exponential tails, and these tails get thinner as we increase $j$; in particular, the probability that $C_{j,i}-m = 0$ approaches one exponentially fast as we increase $j$. However, things are not so easy, since we always condition on $\Pa_N$, and in this space some of the $P_j$'s might be large. This brings us to our general proof strategy. First of all, we will study our probability space conditioned on $\Pa_N$; in particular, in Corollary~\ref{coro:ex_P_1_cond_P_N_wo_weights} and Lemma~\ref{lem:P_j_cond_P_N_upper_bound_wo_weights} below we describe the joint distribution of $P_1, \dots, P_N$ given $\Pa_N$. More specifically, these results show that the $P_j$'s are (more or less) distributed like Poisson random variables with bounded expectations. This will allow us then in Lemma~\ref{lem:P_1=p_R_ge_r_cond_P_N_wo_weights} to show that $L$ dominates the sum $L+R$ in the sense that $\pr{ L + R = n-mN \mid \Pa_N} \sim \pr{ L = n-mN \mid \Pa_N}$ as $n,N,n-N\to\infty$. Subsequently, in Lemma~\ref{lem:probability_E_n_P_N_asymptotic_wo_weights} we expoit the subexponentiality and establish that this last probability is essentially a multiple of $\pr{C_{1,1} = n-mN}$. Just as a side remark and so as to make the notation more accessible: it is instructive to think of the random variable $L$ as something (that will turn out to be) large, and $R$ as some remainder (that will turn out to be small with exponential tails).

Our first aim is to study the distribution -- in particular the tails -- of $P$ and $P^{(\ell)}$, that is, we want to estimate the probability of $\Pa_N$. To this end, consider the probability generating series $F(x)$ and $F^{(\ell)}(x)$ of $P$ and $P^{(\ell)}$, respectively, that is
\[
	F^{(\ell)}(x)
	= \frac1{G^{(\ell)}(\rho)} \cdot \e{\sum_{j>\ell} \frac{C(\rho^j)}{j} x^j},
~~\text{where}~~
G^{(\ell)}(\rho) := \e{\sum_{j>\ell} \frac{C(\rho^j)}{j}} 
\]
and $F(x) = F^{(0)}(x)$, $G^{(0)}(\rho) = G(\rho)$. Hence, the distribution of $P^{(\ell)}$ (and $P$) is given by $(\mathrm{Pr}[P^{(\ell)} = N])_{N\ge0} = ([x^N]F^{(\ell)}(x))_{N\ge0}$. In Lemma~\ref{thm:coeff_F_tilde_F_wo_weights} we determine the precise asymptotic behaviour of these probabilities. But first, we need a simple auxiliary statement.
\begin{proposition}
\label{lem:C(x)_approx_wo_weights}
There exists $A>0$ such that, for all $0 < z \le \rho$ and $j \in \Nat$
\[
	1 \le \frac{C(z^j)}{c_m z^{jm}} \le 1 + A z^j.
\]
\end{proposition}
\begin{proof}
The first inequality follows directly from the definition of $C$ and $m$. Note that
\begin{equation}
\label{eq:Ccmzj}
	\frac{C(z^j)}{c_m z^{jm}} 
	\le 1 + z^{j}\frac{1}{c_m} \sum_{k > m} c_k \rho^{jk - j(m+1)}
	= 1 + z^j\frac{\rho^{-2m}}{c_m}  \sum_{k > m} c_k \rho^{jk - j(m+1)+2m}.
\end{equation}
Since $m \ge 1$ and $k \ge m+1$ we obtain that
\[
	jk - j(m+1) +2m
	= j(k - m-1) +2m 
	\ge (k - m-1) + 2m
	= k + m -1 
	\ge k.
\]
Thus, as $\rho < 1$, we obtain from~\eqref{eq:Ccmzj} the claimed bound with $A = C(\rho)\rho^{-2m}/c_m$. 
\end{proof}

\begin{lemma}
\label{thm:coeff_F_tilde_F_wo_weights}
There exist constants $(B^{(\ell)})_{\ell \in \Nat_0}>0$ such that, as $N\to\infty$
\[
	[x^N] F(x) \sim B^{(0)} \cdot N^{c_m-1}\rho^{mN} 
	\quad
	\text{and}
	\quad
	[x^N]F^{(\ell)}(x) \sim B^{(\ell)} \cdot N^{c_m-1}\rho^{mN} ,
	~~
	\ell\in\Nat,
\]
where
\[
	\frac{B^{(\ell)}}{B^{(0)}} = \e{\sum_{1 \le j \le \ell} \frac{C(\rho^j)}{j}} 
		\e{-\sum_{1 \le j \le \ell} \frac{C(\rho^j)}{j}\rho^{-jm}}.
\]
\end{lemma}
\begin{proof}
We split up
\[
	F(x) = \frac{1}{G(\rho)}
	\cdot \e{\sum_{j \ge 1} \frac{c_m\rho^{jm}}{j}x^j}
	\cdot \e{\sum_{j \ge 1} \frac{C(\rho^j) - c_m\rho^{jm}}{j}x^j}
	=: \frac{1}{G(\rho)}
	\cdot A(x)
	\cdot B(x). 
\]
Proposition \ref{lem:C(x)_approx_wo_weights} asserts that $B(x)$ has radius of convergence $\rho_B \ge \rho^{-(m+1)}$. Further,
\[
	A(x) 
	= (1-\rho^mx)^{-c_m},
\]
and the radius of convergence of $A(x)$ is $\rho_A = \rho^{-m}<\rho^{-(m+1)} \le \rho_B$ (since $\rho<1$). Using Lemma~\ref{lem:asymptpowergeom} we obtain that $[x^N]A(x) \sim{N^{c_m-1}}\rho^{mN}/ {\Gamma(c_m)}$  and thus $A(x)$ has property \eqref{eq:subexp2}. From Lemma \ref{lem:coeff_product} we then obtain that
\begin{equation}
\label{eq:prPNexplicit}
	\pr{P = N} = 
	[x^N]F(x) 
	\sim \frac{1}{G(\rho)} B(\rho_A) [x^N]A(x)
	\sim \frac{B(\rho^{-m})}{G(\rho) \Gamma(c_m)}  \cdot {N^{c_m-1}}{}\rho^{mN},
	\quad
	N\to\infty,
\end{equation}
Similarly, for $\ell\in\Nat$
\[
	F^{(\ell)}(x) 
	= \frac{ A(x)}{G^{(\ell)}(\rho)} 
	\cdot \e{\sum_{j \ge 1} \frac{C(\rho^j) - c_m\rho^{jm}}{j}x^j} \e{- \sum_{1 \le j \le \ell} \frac{C(\rho^j)}{j} x^j}
	=:\frac{A(x)}{G^{(\ell)}(\rho)} 
	\cdot B^{(\ell)}(x).
\] 
Since the radius of convergence of $B^{(\ell)}(x)$ is again (at least) $\rho^{-(m+1)}$  
\[
	[x^N]F^{(\ell)}(x)
	\sim \frac{B^{(\ell)}(\rho^{-m})}{G^{(\ell)}(\rho)\Gamma(c_m)} \cdot {N^{c_m-1}} \rho^{mN}.
\]
\end{proof}
As an immediate consequence of Lemma \ref{thm:coeff_F_tilde_F_wo_weights} we establish the asymptotic distribution of the random vector $(P_1,\dots,P_\ell)$ conditioned on the event $\Pa_N$ for fixed $\ell\in\Nat$; this will be useful later when we consider the distribution of $L$, cf.~\eqref{eq:def_L_R}. Clearly, the condition $\Pa_N$ makes $P_1,\dots,P_{\ell}$ dependent, but the corollary says that this effect vanishes for large $N$. Moreover, we study the moments of $P_1$ given $\Pa_N$.
\begin{corollary}
	\label{coro:ex_P_1_cond_P_N_wo_weights}
Let $\ell\in\Nat$ and $(p_1,\dots,p_\ell)\in\Nat_0^{\ell}$. Then
\begin{equation}
\label{eq:jointconv}
	\pr{\bigcap_{1 \le j \le \ell} \left\{P_j = p_j\right\} \mid \Pa_N}
	\to \prod_{1 \le j \le \ell} \pr{ \pois{ \frac{C(\rho^j)}{j\rho^{jm}} } = p_j},
	\quad
	N\to\infty.
\end{equation}
Moreover, for any $z \in \mathbb{R}$, as 	$N\to\infty$
\[
	\ex{z^{P_1} \mid \Pa_N}
	\to \ex{z^{\pois{ \frac{C(\rho)}{\rho^{m}} }}}
	= \eul^{\frac{C(\rho)}{\rho^{m}}(z-1)},
\quad 
	\ex{P_1 \mid \Pa_N}
	\to \ex{\pois{ \frac{C(\rho)}{\rho^{m}} }} = {C(\rho)}{\rho^{-m}}.
\]
\end{corollary}
\begin{proof}
Let $s = \sum_{1 \le j \le \ell} jp_j$. Using the definition of conditional probability we obtain readily
\[
	\pr{\bigcap_{1 \le j \le \ell} \left\{P_j = p_j\right\} \mid \Pa_N}
	= \frac{\pr{\bigcap_{1 \le j \le \ell} \left\{P_j = p_j\right\} \cap \{P^{(\ell)} = N - s\}}}{\pr{P = N}}.
\]
Since $P_1, \dots P_\ell, P^{(\ell)}$ are independent, the right-hand size equals
\begin{equation}
\label{eq:auxPell}
	\prod_{1 \le j \le \ell}\pr{P_j = p_j} \cdot [x^{ N - s}]F^{(\ell)}(x) / [x^N]F(x),
\end{equation}
and \eqref{eq:jointconv} follows by applying Lemma~\ref{thm:coeff_F_tilde_F_wo_weights}. We will next show $P_1$ given $\Pa_N$ has exponential moments. Abbreviate $B:=C(\rho)\rho^{-m}$. Note that \eqref{eq:jointconv} (where we use $\ell = 1$) yields for any fixed $K\in\Nat$
\[
	\sum_{0\le k\le K} z^k \pr{P_1 = k \mid \Pa_N}
	\sim \sum_{0\le k \le K} z^k \pr{\pois{B} = k},
	\quad
	N\to\infty.
\]
Let $\eps > 0$. Note that we can choose $K$ large enough such that 
the right hand side differs at most $\eps$ from $\ex{z^{\pois{B}}} = e^{B(z-1)}$. In order finish the proof  we will argue that if $K$ and $N$ are large enough, then $\sum_{K \le k \le N}z^k \pr{P_1=k \mid \Pa_N} < \eps$ as well. First, by Lemma~\ref{thm:coeff_F_tilde_F_wo_weights} 
\[
	z^N\pr{P_1 = N \mid \Pa_N} 
	\le z^N \frac{\pr{P_1 = N}}{\pr{\Pa_N}}
	\sim \frac1A N^{-c_m+1} \frac{(zB)^N}{N!}
	\to 0, 
	\quad
	N\to\infty.
\]
Moreover, according to Lemma~\ref{thm:coeff_F_tilde_F_wo_weights} there exists $A_1>0$ such that $[x^{N-k}]F^{(1)}(x)/[x^N]F(x) \le A_1\cdot (1-k/N)^{c_m-1}\rho^{-mk}$  for all $0 \le k \le N-1$. Then with \eqref{eq:auxPell} we obtain
\begin{align}
\label{eq:bounded_sum_poisson}
	\sum_{K\le k \le N-1} z^k \pr{P_1 = k \mid \Pa_N}
	\le A_1 \sum_{K\le k \le N-1} t_k,
	~~
	\text{where}~~
	t_k := (1-k/N)^{c_m-1} \frac{(zB)^k}{k!}.
\end{align}
Note that we can choose $K$ large enough such that, say, $t_{k+1} \le t_k/2$ for all $K \le k < N-1$. Then the sum is bounded by $2t_K$, and choosing $K$ once more large enough gives $2t_K < \eps$.
\end{proof}
 Note that Corollary~\ref{coro:ex_P_1_cond_P_N_wo_weights} (only) holds for a fixed $\ell\in\Nat$; it does not tell us anything about $(P_1,\dots,P_{\ell})$ in the case where $\ell$ is not fixed, or, more importantly, when $\ell = N$ (note that $P_{N'} = 0$ for all $N' > N$ if we condition on $\Pa_N$). Regarding this general case, the following statement gives an upper bound for the probability of the event $\bigcap_{1 \le j \le N} \{P_j = p_j\}$ that is not too far from the right-hand side in Corollary~\ref{coro:ex_P_1_cond_P_N_wo_weights}.  For the remainder of this section it is convenient to define 
\[
	\Omega_N 
	:= \Bigl\{ (p_1,\dots,p_N) \in \Nat_0^{N} ~:~ \sum_{1 \le j \le N} j p_j = N \Bigr\},
	\quad
	N \ge 2.
\]	
In what follows we derive a stochastic upper bound for the distribution of $(P_1,\dots,P_N)$ conditioned on $\Pa_N$.
\begin{lemma}
	\label{lem:P_j_cond_P_N_upper_bound_wo_weights}
There exists an $A>0$ such that for all $N$ and all $(p_1,\dots,p_N)\in\Omega_N$ 
\[
	\pr{\bigcap_{1 \le j \le N} \{P_j = p_j\} \mid \Pa_N}
	\le A \cdot N \cdot 
	\prod_{1 \le j \le N} \pr{\pois{\frac{C(\rho^j)}{j\rho^{jm}}} = p_j}.
\]
\end{lemma}
\begin{proof}
Using the definition of conditional probability  and recalling that the $P_j$'s are independent and $P_j \sim \pois{C(\rho^j)/j}$
\begin{align*}
	\textrm{Pr}\Biggl[\bigcap_{1 \le j \le N} & \{P_j = p_j\} \mid \Pa_N\Biggr]
	\le \frac{1}{\pr{ \Pa_N}} 
	\cdot \prod_{1 \le j \le N} \left(\frac{C(\rho^j)}{j}\right)^{p_j} \frac1{ p_j!} \\
	&= \frac{1}{\pr{ \Pa_N}} 
	\cdot \e{\sum_{1 \le j \le N} \frac{C(\rho^j)}{j\rho^{jm}}} \rho^{mN}
	\cdot \prod_{1 \le j \le N} \pr{\pois{ \frac{C(\rho^j)}{j\rho^{jm}} } = p_j}.
\end{align*}
With Lemma \ref{thm:coeff_F_tilde_F_wo_weights} we obtain the existence of $B_1>0$ such that for $N$ large enough
\[
	\pr{ \Pa_N}^{-1} \le B_1 \rho^{-mN} N^{1-c_m}.
\]
By Proposition \ref{lem:C(x)_approx_wo_weights} there exists a constant $B_2>0$ such that	$C(\rho^j)/\rho^{jm} \le c_m	+ B_2 c_m \rho^j$. Consequently, since $\rho \in (0,1)$ there exists $B_3>0$ such that
\[
	\e{\sum_{1 \le j \le N} \frac{C(\rho^j)}{j\rho^{jm}}}
	\le B_3 N^{c_m},
\]
which concludes the proof.
\end{proof}
With this result at hand we are ready to study the distribution of $R$, cf.~\eqref{eq:def_L_R}. As it will be necessary later, we show uniform tails bounds that hold for the joint distribution of $P_1$ and $R$ conditioned on $\Pa_N$.
\begin{lemma}
	\label{lem:P_1=p_R_ge_r_cond_P_N_wo_weights}
There exist  $A>0$ and $0<a<1$ such that
\[
	\pr{P_1 = p, R = r \mid \Pa_N} \le A \cdot a^{p + r},
	\quad
	p,r,N\in\Nat.
\]
\end{lemma}
\begin{proof}
We will prove the claimed bound by showing appropriate bounds for the moment generating function $\ex{e^{\lambda R}\mid \Pa_N}$.  Let us fix any $0 < \lambda < -\log(\rho) / 2$ such that $\rho e^\lambda < 1$. Then $\rho^j \eul^{\lambda j} < \rho$ for all $j\ge 2$. Recall that $\pr{C_{j,i} = k} = {c_k \rho^{jk}}/{C(\rho^j)}, k\in\Nat, j \ge 2, i \ge 1$, see~\eqref{eq:prob_C_j_i=k}. We obtain that 
\[
	\ex{\eul^{\lambda (j(C_{j,i}-m))}}
	= \sum_{s \ge 0} \pr{ C_{j,i}  = s + m} \eul^{\lambda j s}
	= \eul^{-\lambda j m} \frac{C(\rho^j \eul^{\lambda j })}{C(\rho^j)},
	\quad
	 i \ge 1, j \ge 2.
\]
Let $\Omega_{N,p}$ be the set of all  $\mathrm{p}=(p_2,\dots,p_N)\in\Nat_0^{N-1}$ such that $(p,p_2,\dots,p_N)\in\Omega_N$, i.e. $p=N-\sum_{2\le j\le N}jp_j$, and let  $\E_\mathrm{p}$ be the event
\[
	\E_\mathrm{p} := \{P_1 = p\} \cap \bigcap_{2 \le j \le N}\{P_j = p_j\} .
\]
Then by Markov's inequality and the independence of the $C_{j,i}$'s and the $P_j$'s, for any $\mathrm{p} \in \Omega_{N,p}$
\[
	\pr{P_1 = p, R \ge r \mid \E_\mathrm{p}}	
	= \pr{ \eul^{ \lambda{R}} \ge \eul^{\lambda r} \mid \E_\mathrm{p}}
	\le \eul^{-\lambda r} \ex{\eul^{\lambda  {R}} \mid \E_\mathrm{p}}
	= \eul^{-\lambda r} \prod_{j=2}^N 
	\left (\frac{C(\rho^j\eul^{\lambda j})}{C(\rho^j) \eul^{\lambda j m}} \right)^{p_j}.
\]
Abbreviate $\tau_j := C( (\rho\eul^{\lambda})^j)/(\rho \eul^{\lambda})^{jm}$ for $j \in \Nat$. By Lemma \ref{lem:P_j_cond_P_N_upper_bound_wo_weights} there exists $A_1>0$ such that 
\begin{equation}
\label{eq:1Rget|P_N}
\begin{split}
	\text{Pr}[P_1 = p,~& R \ge r \mid \Pa_N]
	= \sum_{\mathrm{p} \in \Omega_{N,p}} \pr{R \ge r \mid \E_\mathrm{p}} 
	\pr{\E_\mathrm{p} \mid \Pa_N } \\
	&\le A_1 \eul^{-\lambda r} N \e{-\sum_{2 \le j \le N} \frac{C(\rho^j)}{j\rho^{jm}}}
	\frac{(C(\rho)/\rho^m)^{p}}{p!}
	\sum_{\mathrm{p}\in\Omega_{N,p}} 
	\prod_{2 \le j \le N} \frac{(\tau_j/j)^{p_j}}{p_j!}.
\end{split}
\end{equation}
With Proposition \ref{lem:C(x)_approx_wo_weights} we find $A_2>0$ with
\[
	\e{-\sum_{2 \le j \le N} \frac{C(\rho^j)}{j\rho^{jm}}}
	\le \e{- c_m \sum_{2 \le j \le N} \frac{1}{j}}
	\le A_2 N^{-c_m}.
\]
Let $H_j \sim \pois{\tau_j/j}$ be independent for $j=2,\dots,N$ and set $\tau:=\text{exp}(\sum_{2\le j\le N}\tau_j /j)$. Moreover, abbreviate $B := C(\rho)/\rho^m$. From~\eqref{eq:1Rget|P_N} we obtain that there is an $A_3 > 0$ such that
\begin{equation}
\label{eq:2Yget|P_N}
\begin{split}
	\pr{P_1 = p, R \ge r \mid \Pa_N}
	& \le A_3 \eul^{-\lambda r} N^{1 - c_m}
	\cdot 
		\tau \cdot 
	\frac{B^{p}}{p!} 
		\sum_{\mathrm{p} \in \Omega_{N,p}} \prod_{j =2}^N \pr{H_j = p_j}.
\end{split}
\end{equation}
Note that
\[
	\sum_{\mathrm{p} \in \Omega_{N,p}} \prod_{j=2}^N \pr{H_j = p_j}
	= \pr{\sum_{j=2}^N j H_j = N-p}
	= \tau^{-1} \cdot [x^{N-p}] \e{\sum_{j \ge 2}\frac{\tau_j}{j}x^j  }.
\]
Observe that in the last expression we actually have to restrict the summation to the interval $2 \le j \le N$; however, $[x^M] \text{exp}(\sum_{j \ge 2}{\tau_j}x^j/j) = [x^M] \text{exp}(\sum_{2 \le j \le M}{\tau_j}x^j/j)$ for all $M \in \Nat$.
Then
\begin{align*}
	\e{\sum_{j \ge 2}\frac{\tau_j}{j}x^j  }
	&=  \e{c_m \sum_{j \ge  1} \frac{x^j}{j}} \cdot \e{- c_mx + \sum_{j \ge 2}\frac{x^j}{j}(\tau_j - c_m)  } 
	=: G(x) \cdot  H(x).
\end{align*}
By Proposition \ref{lem:C(x)_approx_wo_weights} there exists a constant $A_4>0$ such that
$
	\tau_j
	\le c_m (1 + A_4 (\rho\eul^{\lambda})^j).
$
With this at hand we deduce
that $H(x)$ has radius of convergence (at least) $(\rho\eul^{\lambda})^{-1}$, which by our choice of $\lambda$ is $> 1$. Note that $G(x) = (1-x)^{-c_m}$, which shows together with Lemma~\ref{lem:asymptpowergeom} that $G$ has property \eqref{eq:subexp2} with radius of convergence $1$. 
As $G(x)$ only has positive coefficients, by Lemma \ref{lem:coeff_product} and the remark in~\eqref{eq:asimbuniform} there is an $A_5>0$ such that 
\[
	[x^{N-p}]G(x)H(x)
	\le A_5 (N-p)^{c_m-1},
	\quad
	p=0,\dots,N-1.
\]
All in all,
\begin{align}
	\label{eq:sum_H_j_upper_bound}
	\pr{\sum_{2 \le j \le N} j H_j = N-p} \le A_5 \tau^{-1} (N-p)^{c_m-1},
	\quad
	p=0,\dots,N-1.
\end{align}
For the case $p=N$ note that the probability that $\sum_{2 \le j \le N} j H_j = 0$ equals $\tau^{-1}$.
Putting the pieces together, we get from~\eqref{eq:2Yget|P_N} that there is an $A_6 > 0$ such that
\begin{equation}
\label{eq:prpr|Nlast}
	\pr{P_1 = p, R \ge r \mid \Pa_N}
	\le A_6\eul^{-\lambda r} N^{1-c_m} 
	\left( \frac{B^N}{N!} +  \frac{B^{p}}{p!}  (N-p)^{c_m-1} \cdot \mathbf{1}[p \neq N]\right).
\end{equation}
Observe that $N^{1-c_m }B^N/N! \le e^{-\lambda N}$ for $N$ large enough. Additionally, if $N/2 \le p < N$, then for $N$ large enough
\[
	N^{1-c_m }\frac{(e^\lambda B)^{p}}{p!}  (N-p)^{c_m-1}
	= (1-p/N)^{c_m-1} \frac{(e^\lambda B)^{p}}{p!}  
	\le N^{1-c_m} \cdot  \frac{(e^\lambda B)^{p}}{p!}
	\le 1
\]
and for $0 \le p \le N/2$
\[
	N^{1-c_m }\frac{(e^\lambda B)^{p}}{p!}  (N-p)^{c_m-1}
	\le \max\{2^{1-c_m},1\} \cdot  e^{e^\lambda B} \pr{\pois{e^\lambda B} = p}
	\le \max\{2^{1-c_m},1\} \cdot  e^{e^\lambda B}
\]
is also bounded.  Plugging these bounds into~\eqref{eq:prpr|Nlast} completes the proof.
\end{proof}
We have just proven that $P_1, R$ have (joint) exponential tails when conditioned on $\Pa_N$. The next lemma is the last essential step towards the proof of Theorem \ref{thm:1_coeff_G(x,y)}, where we estimate $\pr{ \E_n \mid \Pa_N}$. Recall from~\eqref{eq:EZP} that
\[
	\pr{ \E_n \mid \Pa_N} = \pr{L + R = n - mN \mid \Pa_N},
	\quad \text{where} \quad
	L = \sum_{1 \le i \le P_1} (C_{1,i} - m).
\]
\begin{lemma}
\label{lem:probability_E_n_P_N_asymptotic_wo_weights}
	Let $C(x)$ be subexponential. Then
	\[
		\pr{ \E_n \mid \Pa_N} \sim c_{n-m(N-1)} \rho^{n-mN},
		\quad
		n,N,n-mN\to\infty.
	\]
\end{lemma}
\begin{proof}
For the entire proof we abbreviate $\widetilde{N} := n - mN$. Then
\begin{equation}
\label{eq:EnPNdec}
	\pr{ \E_n \mid \Pa_N} = \sum_{p \ge 0}
	\sum_{r \ge 0 } \pr{L = \widetilde{N} - r\mid \Pa_N, P_1 = p, R = r} \pr{P_1 = p, R = r\mid \Pa_N}.
\end{equation}
For brevity, let us write in the remainder
\[
	{\cal D}_{N,p,r} = {\cal P}_N \cap \{P_1 = p\} \cap \{R = r\}
	\quad
	\text{and}
	\quad
	Q_{\widetilde{N}} := \text{Pr}\big[C_{1,1} = \widetilde{N}+m\big] = \frac{c_{n-m(N-1)} \rho^{n-m(N-1)}}{C(\rho)}.
\]
We will show that
\begin{equation}
\label{eq:prLsmallr}
	\text{Pr}\big[L = \widetilde{N} - r~|~ {\cal D}_{N,p,r}\big]
	\sim
	p \cdot Q_{\widetilde{N}}
	\quad
	\text{for}
	\quad
	p, r \in \Nat_0, \text{ as } \widetilde{N} \to \infty.
\end{equation}
Let $a \in (0,1)$ be the constant guaranteed to exist from Lemma~\ref{lem:P_1=p_R_ge_r_cond_P_N_wo_weights}, and choose $\delta > 0$ such that 
$(1+\delta)a < 1$. We will also show that there are $C > 0, N_0 \in \Nat$ such that
\begin{equation}
\label{eq:prLallr}
	\text{Pr}\big[L = \widetilde{N} - r~|~ {\cal D}_{N,p,r}\big]
	\le C (1 + \delta)^{p+r} \cdot Q_{\widetilde{N}}
	~~
	\text{for all}
	~~
	p,r \in \Nat_0, \widetilde{N}  \ge N_0.
\end{equation} 
From the two facts~\eqref{eq:prLsmallr} and~\eqref{eq:prLallr} the statement in the lemma can be obtained as follows. We will assume throughout that $\delta$ is fixed as described above, say for concreteness $\delta = (a^{-1}-1)/2$, and choose an $0 < \eps < 1$ arbitrarily. Moreover, we will fix $K\in \mathbb{N}$ in dependence of $\eps$ only, and we will split the double sum in~\eqref{eq:EnPNdec} in three parts with $(p,r)$ in the sets 
\[
	B_{\le} = \{(p,r): 0 \le p,r\le K\}, 
	\quad
	B_{>,\cdot} = \{(p,r): p > K, r\in \Nat_0\},
	\quad
	B_{\cdot, >} =  \{(p,r): p \in \Nat_0, r > K\}.
\]
We will show that the main contribution to $\pr{\E_n~|~ {\cal P}_N}$ stems from $B_\le$, while the other two parts contribute rather insignificantly. Let us begin with treating the latter parts. 
Observe that  using Lemma~\ref{lem:P_1=p_R_ge_r_cond_P_N_wo_weights}
and~\eqref{eq:prLallr} we obtain that there is a constant $C' > 0$ such that for all 
$r \in \Nat_0$ and $K \ge K_0(\eps)$ 
\[
\begin{split}
		\sum_{p \ge K} \pr{L = \widetilde{N} - r~|~ {\cal D}_{N,p,r}} \pr{P_1 = p, R = r~|~ \Pa_N}
		& \le 
		C' \sum_{p \ge K} (1+\delta)^{p+r} \cdot a^{p+r} \cdot Q \\
		&\le \eps \cdot ((1+\delta)a)^{r} \cdot Q_{\widetilde{N}}.
\end{split}
\] 
Since $(1+\delta)a < 1$, summing this over all $r$ readily yields for $c = (1-(1+\delta)a)^{-1}$ that
\begin{equation}
\label{eq:boundB>.}
\sum_{(p,r) \in B_{>, \cdot}} \pr{L = \widetilde{N} - r~|~ {\cal D}_{N,p,r}} \pr{P_1 = p, R = r~|~ \Pa_N} \le c\eps \cdot Q_{\widetilde{N}}.
\end{equation}
Completely analogously with the roles of $p,r$ interchanged  we obtain that also 
\begin{equation}
\label{eq:boundB.>}
\sum_{(p,r) \in B_{\cdot,>}} \pr{L = \widetilde{N} - r~|~ {\cal D}_{N,p,r}} \pr{P_1 = p, R = r~|~ \Pa_N} \le c\eps \cdot Q_{\widetilde{N}}.
\end{equation}
It remains to handle the part of the sum in~\eqref{eq:EnPNdec} with $p,r\in B_\le$. Using~\eqref{eq:prLsmallr} we infer that
\[
	\sum_{(p,r) \in B_{\le}} \pr{L = \widetilde{N} - r~|~ {\cal D}_{N,p,r}} \pr{P_1 = p, R = r~|~ \Pa_N}
	\sim
	\sum_{(p,r) \in B_{\le}} p \pr{P_1 = p, R = r~|~ \Pa_N} \cdot Q_{\widetilde{N}}.
\]
Using Lemma~\ref{lem:P_1=p_R_ge_r_cond_P_N_wo_weights} once again note that we can choose $K$ large enough such that 
\[
	\sum_{0 \le p \le K} \sum_{r \ge K} p \pr{P_1 = p, R = r~|~ \Pa_N} 
	\le A \sum_{0 \le p \le K} \sum_{r \ge K} p a^{p+r}
	\le \eps
\]
and that 
\[
	\left|
		\sum_{p \ge 0} p \pr{P_1 = p~|~ \Pa_N} 
		- \sum_{0 \le p \le K} p \pr{P_1 = p~|~ \Pa_N}
	\right|
	= 	\left|
		\sum_{p > K} p \pr{P_1 = p~|~ \Pa_N} 
	\right|
	\le \eps.
\]
Altogether this establishes that
\[
	\left|
		\sum_{(p,r) \in B_{\le}}
			\pr{L = \widetilde{N} - r~|~ {\cal D}_{N,p,r}}
			\pr{P_1 = p, R = r~|~ \Pa_N}
			- \ex{P_1 \mid {\cal P}_N }Q_{\widetilde{N}}
	\right| \le 2 \eps Q_{\widetilde{N}}.
\]
Corollary~\ref{coro:ex_P_1_cond_P_N_wo_weights} asserts that $\ex{P_1 \mid {\cal P}_N } \to C(\rho)\rho^{-m}$. Since $\eps > 0$ was arbitrary, combining this with~\eqref{eq:boundB>.} and~\eqref{eq:boundB.>} we obtain from~\eqref{eq:EnPNdec} that $\pr{\E_n \mid {\cal P}_N} \sim C(\rho)\rho^{-m} \cdot Q_{\widetilde{N}}$, which is the claim of the lemma.

In order to complete the proof it remains to show the two claims~\eqref{eq:prLsmallr} and~\eqref{eq:prLallr}. We begin with~\eqref{eq:prLsmallr}. Note that for $p, r \in \Nat_0$
\begin{equation}
\label{eq:LPPRsimple}
	\text{Pr}\big[L = \widetilde{N} - r\mid \Pa_N, P_1 = p, R = r\big]
	= \pr{\sum_{1 \le i \le p} C_{1,i} = \widetilde{N}-r+pm}.
\end{equation}
Recall that $\text{Pr}[C_{1,1} = k] = c_k \rho^k / C(\rho)$, where $\rho$ is the radius of convergence of $C$. Since $C$ is subexponential, $c_{k-1} \sim \rho c_k$ and thus the distribution of the~$C_{1,i}$'s is also subexponential with $\text{Pr}[C_{1,1} = k-1] \sim \text{Pr}[C_{1,1} = k]$. 
We obtain with Lemma~\ref{lem:subexp_several_prop}~\ref{lem:random_stopped_sum} that the latter probability is $\sim p \text{Pr}[C_{1,1} = \widetilde{N} - r + pm]$, as $\widetilde{N} \to \infty$.
Moreover, as $\widetilde{N} \to \infty$, $\text{Pr}[C_{1,1} = \widetilde{N} - r + pm] \sim Q_{\widetilde{N}}$, and \eqref{eq:prLsmallr} is established.

We finally show~\eqref{eq:prLallr}. Our starting point is again~\eqref{eq:LPPRsimple}. Note that with Lemma~\ref{lem:subexp_several_prop}~\ref{lem:long_sum_exp} there are $C > 0$ and $N_0\in\Nat$ such that the sought probability is at most $C(1+\delta)^p \text{Pr}[C_{1,1} = \widetilde{N} - r + pm]$ for all $\widetilde{N} - r + pm \ge N_0$. Moreover, as we have argued in the previous paragraph, the distribution of $C_{1,1}$ is subexponential with $\text{Pr}[C_{1,1} = k-1] \sim \text{Pr}[C_{1,1} = k]$; we thus may choose $C$ and $N_0$ large enough such that in addition $\text{Pr}[C_{1,1} = \widetilde{N} - r + pm] \le C (1+\delta)^r Q_{\widetilde{N}}$. This establishes~\eqref{eq:prLallr} if $\widetilde{N} - r + pm \ge N_0$. To treat the remaining cases, note that in this situation we have $r > \widetilde{N}-N_0$. Since the probability generating series of~$C_{1,1}$ is subexponential with radius of convergence 1, we obtain that~$C(1+\delta)^r Q_{\widetilde{N}} > 1$ for sufficiently large~$\widetilde{N}$; thus~\eqref{eq:prLallr} is trivially true in this case.
\end{proof}

With all these facts at hand the proof of Theorem \ref{thm:1_coeff_G(x,y)} is straightforward. With Lemma \ref{lem:coeff_in_prob_wo_weights} and \ref{thm:coeff_F_tilde_F_wo_weights} (in particular, Equation~\eqref{eq:prPNexplicit}) we obtain as $n,N,n-mN\to\infty$,
\begin{align*}
	[x^n y^N]G(x,y)
	&=\frac1{G(\rho)} \rho^{-n} \pr{\E_n \mid \Pa_N} \pr{\Pa_N}  \\
	&\sim \frac{1}{\Gamma(c_m)} \e{\sum_{j \ge 1} \frac{C(\rho^j)-c_m\rho^{jm}}{j} \rho^{-jm}} 
	N^{c_m-1} c_{n-m(N-1)} .
\end{align*}

\subsection{Proof of Theorem \ref{thm:3_size_largest_comp_L_n_N}}
	\label{subsec:proof_3}

Let us begin with (re-)collecting all basic definitions that will be needed in the proof. Suppose that $C(x)$ is subexponential with radius of convergence $0 < \rho < 1$ and set $m := \min\{k \in \Nat: c_k > 0\}$, see also~\eqref{eq:mc}. Moreover, let $P_j \sim\pois{C(\rho^j)/j}, j\in\Nat$ and $C_{j,1},\dots,C_{j,P_j}, {j\in\Nat}$ have the distribution specified in~\eqref{eq:prob_C_j_i=k}, that is, $\pr{C_{j,i} = k} = {c_k \rho^{jk}}/{C(\rho^j)}, k,i,j\in\Nat$. We assume that all these random variables are independent.
 Let $\Pa_N$, $\E_n$ be as in~\eqref{eq:def_E_n_P_N}, that is, with
 \[
	P = \sum_{j\ge 1} j P_j,
	\quad
	L = \sum_{1 \le i \le P_1} (C_{j,i} - m),
	\quad R = \sum_{j\ge 2} j \sum_{1 \le i \le P_j} (C_{j,i}-m)
 \]
we have that $\Pa_N = \{P=N\}$ and $\E_n = \{L+R = n-mP\}$.
 
With this notation at hand, let  $\mathsf{G}_{n,N}$ be a uniformly drawn random object from ${\cal G}_{n,N}$, meaning that the number of atoms is $n$ and the number of components $N$. According to Lemma \ref{lem:algo_same_distri_boltzmann} and using that the Boltzmann model induces the uniform distribution on objects of the same size, we infer that 
\[
	\pr{\mathsf{G}_{n,N} = G}
	= \frac1{|{\cal G}_{n,N}|}
	= \frac{\rho^n / C(\rho)}{|{\cal G}_{n,N}| \rho^n / C(\rho)}
	= \frac{\pr{\Lambda G = G}}{\pr{\Pa_N, \E_n}}
	= \pr{\Lambda G = G \mid \Pa_N, \E_n},
	\quad
	G \in {\cal G}_{n,N},
\]
that is, studying the distribution of $\mathsf{G}_{n,N}$ boils down to considering the distribution of $\Lambda G$ conditional on both $\Pa_N, \E_n$.
This is the starting point of our investigations. In particular, $\mathsf{G}_{n,N}$ has $N$ components with sizes given by the vector $(C_{j,i}: 1 \le j \le N, 1 \le i \le P_j)$. Our aim is here to study the properties of that vector in the conditional space given by  $\Pa_N, \E_n$. To this end, set
\begin{equation}
	M^{*} := \max_{j\ge 1,1\le i\le P_j}C_{j,i}
	\quad \text{ and } \quad 
	C^{*}_p := \max\{C_{1,1},\dots,C_{1,p}\} \quad \text{for $p\in \Nat$}.
\label{eq:M*C*}
\end{equation}
Then the statement of the theorem is that, conditional on $\Pa_N, \E_n$, we have that $M^* = n-mN + \mathcal{O}_p(1)$; since the total number of atoms is $n$, the number of components is $N$, and the smallest component contains $m$ atoms, this immediately implies that there are $N+\mathcal{O}_p(1)$ components with exactly $m$ atoms, and all remaining components have a total size of $\mathcal{O}_p(1)$ as well.

The general proof strategy in the remaining section is as follows. We first show in Lemma~\ref{lem:P_1=p_R_ge_r_cond_E_n_P_N_wo_weights} that both $P_1, R$ are ``small'' in the conditioned space; this makes sure that only a bounded number of entries in the vector $(C_{j,i})_{j\ge 2,1\le i\le P_j}$ are larger than $m$, and that this total excess is bounded. Hence, the remaining number of $n - (N-P_1)m + \mathcal{O}_p(1)$ atoms is to be found in the components with sizes in $(C_{1,i})_{1\le i\le P_1}$. In Lemma~\ref{lem:P_1=p_R_ge_r_cond_E_n_P_N_wo_weights} we exclude that $P_1$ grows too large conditioned on $\E_n,\Pa_N$; indeed,  we show that it is stochastically bounded. Then the property of subexponentiality guarantees that only the maximum of the $C_{1,i}$'s dominates the entire sum, cf. Lemma~ \ref{lem:subexp_several_prop}~\ref{lem:big_jump_principle}, and Theorem \ref{thm:3_size_largest_comp_L_n_N} follows. 

Let us now fill this overview with details. Recall Lemma~\ref{lem:P_1=p_R_ge_r_cond_P_N_wo_weights}, which says that $P_1,R$ have (joint) exponential tails given $\Pa_N$. We show that conditioning in addition to $\E_n$ does not change the behaviour qualitatively. The proof can be found at the end of the section.
\begin{lemma}
	\label{lem:P_1=p_R_ge_r_cond_E_n_P_N_wo_weights}
There exist constants $A>0$ and $0<a<1$ such that
\[
	\pr{P_1=p, R = r \mid \E_n, \Pa_N}
	\le A \cdot a^{p+r},
	\quad
	p,r,n,N\in\Nat.
\]
\end{lemma}
With this lemma the proof of the theorem can be completed as follows. Let $\eps > 0$ be arbitrary. Abbreviate $\widetilde{N} = n-mN$. With $M^*$ as in~\eqref{eq:M*C*} we will show that there is $K\in\Nat$ such that
\[
	\pr{ | M^{*} - \widetilde{N} | \ge K \mid \E_n, \Pa_N}
	< \eps
\]
for $n,N,\widetilde{N}$ sufficiently large, which is the statement of the theorem. According to Lemma~\ref{lem:P_1=p_R_ge_r_cond_E_n_P_N_wo_weights} 
there exist constants $C_R,C_P\in\Nat$ such that 
\[
	\pr{ R \ge C_R, P_1 \ge C_P \mid \E_n, \Pa_N } < \eps/2,
	\quad
	n,N,\widetilde{N} \in \Nat.
\]
We deduce 
\begin{equation}
	\pr{ | M^{*} - \widetilde{N} | \ge K \mid \E_n, \Pa_N}
	\le \frac{\eps}{2} + \sum_{0 \le r \le C_R} \sum_{1 \le p \le C_P} 
	\pr{ { |M^{*} - \widetilde{N}| } \ge K \mid \E_n , \Pa_N, R=r, P_1 = p}.
\label{eq:M*geK}
\end{equation}
Note that we only need to consider values of $p$ which are larger than $1$ as $p=0$ excludes $R=r\le C_R < \widetilde{N}$.
The event ``$\E_n , \Pa_N, R=r, P_1 = p$'' implies that $\card{C_{j,i}} \le m+r$ for all $j\ge 2,1\le i \le P_j,$ and $S_p := \sum_{1 \le i \le p} C_{1,i} = \widetilde{N} - r + pm$.  Recall the definition of $C^*$ from~\eqref{eq:M*C*}. Assume that $C^{*}_p\le m + r$, then we get the contradiction $\widetilde{N}-r+pm=S_p \le p(m+r) < \widetilde{N} - r +pm$ for $\widetilde{N}$ large enough. 
It follows that $C^{*}_p > m+r$ and hence we are allowed to interchange $C^{*}_p$ and $M^{*}$ in this conditioned space. That yields
\[
	\pr{ |M^{*} - \widetilde{N} | \ge K \mid \E_n , \Pa_N, R=r, P_1 = p}
	= \pr{ |C^{*}_p - \widetilde{N} | \ge K \mid S_p = \widetilde{N} - r + pm},
\]
for $1 \le p \le C_P,0\le r \le C_R$.
As $C^{*}_p$ is at most $\widetilde{N} - r + pm$ under this condition, we particularly obtain that $\{C^{*}_p \ge \widetilde{N} + K \} = \emptyset$ for $K \ge mC_P$ as long as $0\le p \le C_P$ and $r\ge0$. Consequently, for $1 \le p \le C_P, 0\le r \le C_R$, 
\[
	\pr{ \lvert C^{*}_p - \widetilde{N} \rvert \ge K \mid S_p = \widetilde{N} - r + pm}
	= \pr{ C^{*}_p \le \widetilde{N} - K \mid S_p = \widetilde{N} - r + pm}.
\]
Now  Lemma~\ref{lem:subexp_several_prop}~\ref{lem:big_jump_principle} is applicable as $C_{1,i}$ has subexponential distribution for $1\le i\le p$ and hence for $1\le p\le C_P,0\le r\le C_R$ we have $(C^{*}_p \mid S_p =\widetilde{N} + r -pm) = \widetilde{N} + r -pm + \mathcal{O}_p(1)$ as $\tilde{N}\to\infty$. Consequently, choosing $K$ large enough, 
\[
	\pr{ C^{*}_p \le \widetilde{N} - K \mid S_p = \widetilde{N} - r + pm}
	< \frac{\eps}{2C_RC_P},
	\quad 1\le p\le C_P ,0\le r\le C_R.
\]
We conclude from~\eqref{eq:M*geK}
\begin{align*}
	\pr{|M^{*} - \widetilde{N}| \ge K \mid \E_n, \Pa_N}
	\le \frac{\eps}{2} + \sum_{0 \le r \le C_R} \sum_{1 \le p \le C_P}  
	\pr{ C^{*}_p\le \widetilde{N} - K \mid S_p = \widetilde{N} - r + pm} 
	< \eps.
\end{align*}
Since $\eps>0$ was arbitrary we have just proven that the largest component satisfies $(M^{*} \mid \E_n, \Pa_N) = \widetilde{N} + \mathcal{O}_p(1)$, and the proof is completed. 
\begin{proof}[Proof of Lemma \ref{lem:P_1=p_R_ge_r_cond_E_n_P_N_wo_weights}]
We start with the observation
\begin{align}
\label{eq:P1_R_cond_split_up}
	\pr{P_1 = p, R = r \mid \E_n, \Pa_N}
	= \pr{ \E_n \mid P_1 = p, R=r, \Pa_N} \pr{P_1 = p, R=r\mid \Pa_N} \pr{\E_n \mid\Pa_N}^{-1}.
\end{align}
Set $\widetilde{N} := n - mN$ and $L_p := \sum_{1 \le i \le p} (C_{1,i}-m)$ for $p\in\Nat_0$ as well as $Q_{\widetilde{N}} = \text{Pr}[C_{1,1} -m = \widetilde{N}]$. Let $0<a<1$ be the constant from Lemma \ref{lem:P_1=p_R_ge_r_cond_P_N_wo_weights} and let $\delta>0$ be such that $(1+\delta)a<1$. With \eqref{eq:prLallr} we obtain that there exists $A_1>0$ with
\[
	\pr{ \E_n \mid P_1 = p, R=r, \Pa_N}
	= \pr{ L_p = \widetilde{N} - r \mid \Pa_N}
	\le A_1 (1+\delta)^{p+r} Q_{\widetilde{N}},
	\quad p,r,n,N\in\Nat.
\]
Lemma \ref{lem:P_1=p_R_ge_r_cond_P_N_wo_weights} tells us that we find $A_2>0$ with
\[
	\pr{P_1=p,R=r\mid \Pa_N} 
	\le A_2 a^{p+r},
	\quad p,r,N\in\Nat.
\]
Finally, according to Lemma \ref{lem:probability_E_n_P_N_asymptotic_wo_weights} there is a constant $A_3$ such that
\[
	\pr{\E_n\mid\Pa_N}
	\ge A_3 Q_{\widetilde{N}},
	\quad n,N\in\Nat,
\]
and the claim follows with $a$ replaced by $(1+\delta)a < 1$ by plugging everything into \eqref{eq:P1_R_cond_split_up}.
\end{proof}

\subsection{Proof of Theorem \ref{thm:4_distribution_of_remainder_R_n_N}}
	\label{subsec:proof_4}
For the proof of this theorem we use the equivalent definition of multisets in which all objects not occurring in $G\in\Aux$ are counted with multiplicity $d=0$.
Let $G=\{(C,d_C):C\in\C_{>m}\}\cup\{(C,0):C\in\C_m\}\in\Aux$ and assume that $N(n)\equiv N$ is such that $N(n),n-mN(n)\to\infty$ as $n\to\infty$. Let us write $R_{n,N}$ for the object obtained after removing (i.e. setting the multiplicity to $0$) all objects of size $m$ and a largest component (i.e. subtracting the multiplicity by one) from $\mathsf{G}_{n,N}$. The statement of the theorem is equivalent to showing that
\[
	\pr{R_{n,N} = G}
	\to \e{-\sum_{j\ge 1}\frac{C(\rho^j)-c_m\rho^{jm}}{j\rho^{jm}}}
	\rho^{\lvert G\rvert - m \kappa(G)},
	\quad 
	n\to\infty,
\]
see also~\eqref{eq:boltzmannG>m}. Defining the family of multiplicity counting functions $(d_C(\cdot))_{C\in\C}$ by $(d_C(G))_{C\in\C} = (d_C)_{C\in\C}$ for $G=\{(C,d_C):C\in\C\}\in\Aux$ we immediately obtain  that
\[
	\pr{R_{n,N} = G}
	= \pr{ \forall C\in\C_{>m}: d_C(R_{n,N}) = d_C}.
\]
Let $S>\max\{m,\lvert G\rvert\}$ be some arbitrary integer to be specified later. We infer that
\[
	\pr{R_{n,N} = G}
	\le \pr{ \forall C\in\C_{m+1,S}: d_C(R_{n,N}) = d_C}.
\]
To obtain a lower bound, since $S > \lvert G\rvert$, we observe that $\{\forall C\in\C_{>m}: d_C(R_{n,N})=d_C\}$ is the same as $\{\forall C\in\C_{m+1,S}: d_C(R_{n,N})=d_C\}\cap \{\forall C\in\C_{>S}: d_C(R_{n,N}) = 0\}$. Moreover, note that  $\lvert R_{n,N}\rvert\le S$ implies $d_C(R_{n,N}) = 0$ for all $C\in\C_{>S}$. Thus
\begin{align*}
	\pr{R_{n,N} = G}
	&\ge \pr{\forall C\in\C_{m+1,S}: d_C(R_{n,N})=d_C,
	\lvert R_{n,N}\rvert\le S} \\
	&\ge \pr{\forall C\in\C_{m+1,S}: d_C(R_{n,N})=d_C} 
	- \pr{\lvert R_{n,N} \rvert > S}.
\end{align*}
Let $\eps>0$. According to Theorem \ref{thm:3_size_largest_comp_L_n_N} there is $S_1 > \max\{m,\lvert G\rvert\}$ so that $\pr{\lvert R_{n,N} \rvert > S_1}<\eps$. Hence $\pr{R_{n,N} = G}$ differs by at most $\eps$ from $\pr{\forall C\in\C_{m+1,S}:d_C(R_{n,N})=d_C}$ for all $S>S_1$.
Let us write $L_{n,N}$ for the size of a largest component in $\mathsf{G}_{n,N}$. Theorem \ref{thm:3_size_largest_comp_L_n_N} guarantees that $L_{n,N}$ is unbounded whp, and so we obtain for any $S \in \mathbb{N}$
\[
	\pr{\forall C\in\C_{m+1,S}: d_C(R_{n,N})=d_C}
	= \pr{\forall C\in\C_{m+1,S}: d_C(R_{n,N})=d_C, \lvert L_{n,N}\rvert > S} + o(1).
\]
However, the event $\{\forall C\in\C_{m+1,S}:d_C(R_{n,N})=d_C, \lvert L_{n,N}\rvert > S\}$ is equivalent to the event $\{\forall C\in\C_{m+1,S}:d_C(\mathsf{G}_{n,N})=d_C, \lvert L_{n,N}\rvert > S\}$, since we obtain $R_{n,N}$ by removing all components with size $m$ and a largest component (of size $>S$) from $\mathsf{G}_{n,N}$. Now we add and subtract $\pr{\forall C\in\C_{m+1,s}:d_C(\mathsf{G}_{n,N})=d_C, \lvert L_{n,N} \rvert \le S} = o(1)$ in order to get rid of the event $\lvert L_{n,N}\rvert >S$ and arrive at the fact
\[
	\pr{\forall C\in\C_{m+1,S}: d_C(R_{n,N})=d_C}
	= \pr{\forall C\in\C_{m+1,S}: d_C(\mathsf{G}_{n,N})=d_C} + o(1).
\]
Combining all previous facts yields that for $n$ sufficiently large
\begin{equation}
\label{eq:prRnNG}
	\big|\pr{R_{n,N} = G} - \pr{\forall C\in\C_{m+1,S}: d_C(\mathsf{G}_{n,N})=d_C}\big| \le 2\eps
\end{equation}
and thus we are left with estimating $\pr{\forall C\in\C_{m+1,S}: d_C(\mathsf{G}_{n,N})=d_C}$.
For $\boldrm{v}_S := (v_C)_{C\in\C_{m+1,S}}$ denote by $G(x,y,\boldrm{v}_S)$ the generating series of $\Aux$ such that $x$ marks the size, $y$ the number of components and $\boldrm{v}_S = (v_C)_{C\in\C_{m+1,S}}$ the multiplicities of $(C)_{C\in\C_{m+1,S}}$, or in other words: for $\ell,k\in\Nat_0, \boldrm{t}_S := (t_C)_{C\in\C_{m+1,S}} \in \Nat_0^{\lvert\C_{m+1,S}\rvert}$ the coefficients are given by
\[
	g_{\ell,k,\boldrm{t}_S}
	= \lvert \{ G\in\Aux : \lvert G \rvert = \ell, \kappa(G)=k,\forall C\in\C_{m+1,S}:d_C(G) = t_C\} \rvert.
\]
Setting $v_C=1$ for all $C\in\C_{m+1,S}$ we obtain the generating series $G(x,y)$ counting only size and number of components by $x$ and $y$ respectively. As $\mathsf{G}_{n,N}$ is drawn uniformly at random from $\Aux_{n,N}$ the proof reduces to determining 
\[
	\pr{\forall C\in\C_{m+1,S}: d_C(\mathsf{G}_{n,N})=d_C}
	= \frac{[x^n y^N \boldrm{v}_s^{\boldrm{d}_S}]G(x,y,\boldrm{v}_S)}{[x^n y^N]G(x,y)}.
\]
The following lemma, whose proof is shifted to the end of this section, accomplishes this task.
\begin{lemma}
\label{lem:coefficients_G(x,u,v)}
Let $\boldrm{d} = (d_C)_{C\in\C_{m+1,S}}$ with $D:=\sum_{C\in\C_{m+1,S}}\lvert C\rvert d_C$ and $D':=\sum_{C\in\C_{m+1,S}}d_C$. Then
\[
	\frac{[x^n y^N \boldrm{v}_s^{\boldrm{d}_S}]G(x,y,\boldrm{v}_S)}{[x^n y^N]G(x,y)}
	\to \rho^{D-mD'} 
	\prod_{C\in\C_{m+1,S}}(1-\rho^{\lvert C\rvert-m}),
	\quad
	n\to\infty.
\]
\end{lemma}
\noindent
Lemma \ref{lem:coefficients_G(x,u,v)} yields directly for sufficiently large $n$
\[
	\biggl\lvert\pr{\forall C\in\C_{m+1,S}:d_C(\mathsf{G}_{n,N})=d_C}
	-\rho^{\lvert G\rvert-m\kappa(G)} 
	\prod_{C\in\C_{m+1,S}}(1-\rho^{\lvert C\rvert-m}) \biggr\rvert
	<\eps.
\]
Now observe that with defining $C_{m+1,S}(x) := \sum_{m<\ell\le S}\lvert\C_\ell\rvert x^\ell$ we obtain 
\[
	\lim_{S\to\infty} \prod_{C\in\C_{m+1,S}}(1-\rho^{\lvert C\rvert -m})
	= \lim_{S\to\infty} \prod_{m < \ell \le S} \e{|\C_\ell| \log(1 - \rho^{\ell-m})}
	= \lim_{S\to\infty} \e{-\sum_{j\ge 1} \frac{C_{m+1,S}(\rho^j)}{j\rho^{jm}}}.
\]
By the continuity of $\e{\cdot}$ and  monotone convergence this equals $G_{>m}(\rho)^{-1}$. Choose $S_2>\max\{m,\lvert G\rvert\}$ large enough such that $\prod_{C\in\C_{m+1,S}}(1-\rho^{\lvert C\rvert -m})$ differs at most by $\eps$ from $G_{>m}(\rho)^{-1}$ for all $S>S_2$. Summarizing, fixing $S\ge \max\{S_1,S_2\}$ we obtain for sufficiently large $n$
\[
	\big| \pr{\forall C\in\C_{m+1,S}:d_C(\mathsf{G}_{n,N})=d_C} - \rho^{\lvert G\rvert-m\kappa(G)} G_{>m}(\rho)^{-1} \big| \le 2\eps.
\]
Since $\eps>0$ was arbitrary the proof of the theorem is finished with~\eqref{eq:prRnNG}.
\begin{proof}[{Proof of Lemma \ref{lem:coefficients_G(x,u,v)}}]
First we determine $G(x,y,\boldrm{v}_S)$ explicitly. Define the multivariate generating series 
\[
	C(x,y,\boldrm{v}_S)
	= y \left( C(x) + \sum_{C\in\C_{m+1,S}}(v_C-1)x^{\lvert C\rvert} \right),
\]
where as usual $x$ marks the size, $y$ the number of components (which by convention is always $1$ for $C\in\C$) and $\boldrm{v}_S$ objects in $\C_{m+1,S}$. Note that these parameters are clearly additive when forming multisets.
Hence, according to \cite[Theorem III.1]{Flajolet2009} the formula \eqref{eq:ogf_G_wo_weights} extends to the multivariate version
\begin{align}
\label{eq:G(x,y,v)=e(C(x,y,v))}
	G(x,y,\boldrm{v}_S)
	= \e{\sum_{j\ge 1}\frac{C(x^j,x^j,\boldrm{v}_S^j)}{j}},
\end{align}
where $\boldrm{v}_S^j = (v_C^j)_{C\in\C_{m+1,S}}$. Setting $v_C = 1$ for all $C\in\C_{m+1,S}$ we see that $G(x,y,\boldrm{1}) \equiv G(x,y)$ such that $[x^ny^N]G(x,y) = \lvert \Aux_{n,N}\rvert$. By elementary algebraic manipulations we reformulate \eqref{eq:G(x,y,v)=e(C(x,y,v))} to
\begin{align}
	\begin{split}
	\label{eq:G(x,y,v)_explicit}
		G(x,y,\boldrm{v}_S)
		&= G(x,y) \e{\sum_{C\in\C_{m+1,S}}\left(\sum_{j\ge 1}\frac{(x^{\lvert C\rvert}yv_C)^j}{j} 
		-\sum_{j\ge 1} \frac{(x^{\lvert C\rvert}y)^j}{j}\right)} \\
		&= G(x,y) \prod_{ C\in\C_{m+1,S}}
		\frac{1-x^{\lvert C\rvert}y}{1-x^{\lvert C\rvert} yv_C}.
	\end{split}
\end{align}
Let us now turn to the initial claim in Lemma \ref{lem:coefficients_G(x,u,v)}.
We obtain that
\begin{align*}
	[x^n y^N \boldrm{v}_S^{\boldrm{d}_S}]G(x,y,\boldrm{v}_S)
	&=[x^n y^N]G(x,y) \prod_{C\in\C_{m+1,S}} 
	[v_C^{d_C}] \frac{1-x^{\lvert C\rvert}y}{1-x^{\lvert C\rvert}v_C y} \\
	&= [x^{n-D} y^{N-D'}]G(x,y)  \prod_{C\in\C_{m+1,S}}(1-x^{\lvert C\rvert}y).
\end{align*}
Since $\C_{m+1,S}$ does only have finitely many elements, there exist $L,K\in\Nat$ such that $[x^\ell y^k]\prod_{C\in\C_{m+1,S}}(1-x^{\lvert C\rvert}y) = 0$ for all $\ell \ge L,k\ge K$. Recall that, using Theorem \ref{thm:1_coeff_G(x,y)},
\[
	[x^ny^N]G(x,y) 
	\sim \e{\sum_{j\ge 1}\frac{C(\rho^j)-c_m\rho^{jm}}{j\rho^{jm}}}
	\frac{N^{c_m-1}}{\Gamma(c_m)} \lvert\C_{n-m(N-1)}\rvert,
	\quad
	n\to\infty,
\]
and so $[x^{n-a}y^{N-b}]G(x,y) \sim [x^ny^N]G(x,y)\rho^{a-mb}$ for fixed $a,b\in\Nat$ as $\C$ is subexponential. 
Hence, as $n\to\infty$,
\begin{align*}
	[x^n y^N \boldrm{v}_S^{\boldrm{d}_S}]G(x,y,\boldrm{v}_S)
	&= \sum_{\ell\in[L],k\in[K]} [x^{n-D-\ell}y^{N-D'-k}] G(x,y)
	[x^\ell y^k]\prod_{C\in\C_{m+1,S}}(1-x^{\lvert C\rvert} y) \\
	&\sim [x^ny^N]G(x,y)\cdot \rho^{D-mD'} 
	\sum_{\ell\in[L],k\in[K]} \rho^{\ell-mk} 
	[x^\ell y^k]\prod_{C\in\C_{m+1,S}}(1-x^{\lvert C\rvert} y) \\
	&= [x^ny^N]G(x,y)\cdot \rho^{D-mD'} \prod_{C\in\C_{m+1,S}}(1-\rho^{\lvert C\rvert-m}),
\end{align*}
which finishes the proof.
\end{proof}

\subsection{Proof of Proposition \ref{prop:BS_limit}}
\label{subsec:proof_BS}
\begin{proof}[Proof of Proposition \ref{prop:BS_limit}]
It is a well-known fact that the weak convergence of $(\mathsf{G}_n,o_n)$ to $(\mathbbmss{G},\mathbbmss{o})$ in \eqref{eq:BS_def} is equivalent to showing that for any bounded and continuous function $f:\B_*\to\Real$
\[
	\lim_{n\to\infty} \ex{f(\mathsf{G}_n,o_n)}
	= \ex{f(\mathbbmss{G},\mathbbmss{o})}.
\]
For any finite graph $G$ denote by $o_G$ a vertex chosen uniformly at random from its vertex set. Let $\M(\mathsf{G}_{n,N})$ denote a (canonically chosen) largest component of $\mathsf{G}_{n,N}$ and $\R(\mathsf{G}_{n,N})$ the remainder after removing all objects of size $m$ and $\M(\mathsf{G}_{n,N})$. Let $f:\B_*\to\Real$ be an arbitrary bounded and continuous function. Then
\begin{align*}
	\ex{f(\mathsf{G}_{n,N},o_n)}
	=& \ex{f(\M(\mathsf{G}_{n,N}),o_{\M(\mathsf{G}_{n,N})})}\pr{o_n\in\M(\mathsf{G}_{n,N})} \\
	+&\ex{f(\R(\mathsf{G}_{n,N}),o_{\R(\mathsf{G}_{n,N})})}\pr{o_n\in\R(\mathsf{G}_{n,N})} \\
	+&\ex{f(\mathsf{C}_m,o_m)}\pr{o_n\notin \R(\mathsf{G}_{n,N})\cup\M(\mathsf{G}_{n,N})}.
\end{align*}
According to Theorem \ref{thm:3_size_largest_comp_L_n_N} we have that $\lvert\M(\mathsf{G}_{n,N})\rvert=n-mN+\mathcal{O}_p(1)$ implying $\pr{o_n\in\M(\mathsf{G}_{n,N})}\sim (n-mN)/n \to 1-\lambda$. As the size of $\M(\mathsf{G}_{n,N})\in\C$ tends to infinity and $(\mathsf{C}_n)_{n\ge 1}$ converges in the BS sense to $(\mathbbmss{C},\mathbbmss{o})$ we have that  
\[
	\ex{f(\M(\mathsf{G}_{n,N}),o_{\M(\mathsf{G}_{n,N})})}\pr{o_n\in\M(\mathsf{G}_{n,N})}
	\to (1-\lambda) \ex{f(\mathbbmss{C},\mathbbmss{o})},
	\quad n,N\to\infty.
\]
Theorem \ref{thm:4_distribution_of_remainder_R_n_N} entails that $\R(\mathsf{G}_{n,N})$ has a limiting distribution and hence $\pr{o_n\in\R(\mathsf{G}_{n,N})}\to 0$. As $f$ is bounded 
\[
	\ex{f(\R(\mathsf{G}_{n,N}),o_{\R(\mathsf{G}_{n,N})})}\pr{o_n\in\R(\mathsf{G}_{n,N})}
	\to 0,
	\quad
	n\to\infty.
\]
Finally, we obtain by combining Theorems \ref{thm:3_size_largest_comp_L_n_N} and \ref{thm:4_distribution_of_remainder_R_n_N} that $n - \lvert \R(\mathsf{G}_{n,N})\cup\M(\mathsf{G}_{n,N})\rvert = mN+\mathcal{O}_p(1)$ and hence $\pr{o_n\notin \R(\mathsf{G}_{n,N})\cup\M(\mathsf{G}_{n,N})} \sim mN/n \to \lambda$. Thus, 
\[
	\lim_{n,N\to\infty}\ex{f(\mathsf{G}_{n,N},o_n)}
	= (1-\lambda) \ex{f(\mathbbmss{C},\mathbbmss{o})}
	+ \lambda \ex{f(\mathsf{C}_m,o_m)}.
\]
\end{proof}

\section*{Acknowledgements}
The authors thank Benedikt Stufler for fruitful discussions and valuable input to the proof of Theorem \ref{thm:3_size_largest_comp_L_n_N}.
\bibliographystyle{abbrvnat}
\bibliography{references}

\end{document}